\documentclass[11pt,a4paper]{amsart}

\usepackage{lmodern}
\usepackage[T1]{fontenc}
\usepackage[utf8]{inputenc}
\usepackage{luainputenc}
\usepackage[english]{babel} %

\usepackage{etoolbox}
\patchcmd{\subsection}{-.5em}{.5em}{}{}

\usepackage{array}		  %
\usepackage{graphicx}   %
\usepackage{color}   %
\usepackage{enumitem}   %
\usepackage[hypertexnames=false,pdftex,backref=page,
 	pdfpagemode=UseNone,
 	breaklinks=true,
 	extension=pdf,
 	colorlinks=true,
 	linkcolor=blue,
 	citecolor=red,
 	urlcolor=blue,
 ]{hyperref}

\usepackage[capitalise,noabbrev]{cleveref}
\usepackage{ifthen}     %
\usepackage{todonotes}

\usepackage{datetime}

\usepackage{amsmath}
\usepackage{amssymb}
\usepackage{mathrsfs}   %
\usepackage{mathtools}	%
\mathtoolsset{centercolon}	%
\usepackage{tikz-cd}    %
\usetikzlibrary{babel}  %
\usepackage{stmaryrd}   %
\usepackage{relsize}   %
\usepackage{txfonts}		%

\usepackage{tikz}				%
\usetikzlibrary{calc} %
\usetikzlibrary{arrows}	%

\newcommand\restr[2]{{\left.\kern-\nulldelimiterspace #1 \right|_{#2}}}

\newcommand\supind[1]{{\smash[t]{(#1)}}}

\let\originalleft\left
\let\originalright\right
\renewcommand{\left}{\mathopen{}\mathclose\bgroup\originalleft}
\renewcommand{\right}{\aftergroup\egroup\originalright}

\hypersetup{bookmarksdepth=3}

\numberwithin{equation}{section}
\numberwithin{figure}{section}

\newtheorem{thm}{Theorem}[section]
\newtheorem{theorem}[thm]{Theorem}
\newtheorem{assumption}[thm]{Assumption}
\newtheorem{prop}[thm]{Proposition}
\newtheorem{lemma}[thm]{Lemma}
\newtheorem{cor}[thm]{Corollary}
\newtheorem{corollary}[thm]{Corollary}
\newtheorem{conj}[thm]{Conjecture}
\newtheorem*{conj*}{Conjecture}

\newtheorem{problem}[thm]{Problem}
\newtheorem*{question*}{Question}
\theoremstyle{definition}
\newtheorem{definition}[thm]{Definition}
\theoremstyle{definition}
\newtheorem{rmk}[thm]{Remark}

\newtheorem{example}[thm]{Example}
\newtheorem{notation}[thm]{Notation}

\clearpage{}%

\newcommand{\slocus}{{\scrA}} 	%
\DeclareMathOperator{\hrk}{rank} 	%
\newcommand{\wdeg}{{\textrm{deg}_w}} 	%

\DeclareMathOperator{\gr}{gr}

\newcommand{\suchthat}{\;\ifnum\currentgrouptype=16 \middle\fi|\;}
\newcommand{\bigmid}{\left.\vphantom{\Big\{} \suchthat \vphantom{\Big\}}\right.}

\DeclareMathOperator{\ch}{Char}

\newcommand{\scrA}{{\mathscr A}}

\newcommand{\A}{{\mathbb A}}

\newcommand{\C}{{\mathbb C}}

\newcommand{\N}{{\mathbb N}}

\newcommand{\Q}{{\mathbb Q}}
\newcommand{\R}{{\mathbb R}}

\newcommand{\Z}{{\mathbb Z}}

\newcommand{\abs}[1]{{\left|#1\right|}}

\newcommand{\Ann}{\operatorname{Ann}}

\newcommand{\del}{\partial}

\newcommand{\diag}{\operatorname{diag}}
\newcommand{\di}{\partial}

\newcommand{\GL}{\operatorname{GL}}

\renewcommand{\iff}{\Leftrightarrow}

\newcommand{\isom}{{\ \cong\ }}

\newcommand{\ohne}{{\ \setminus \ }}

\newcommand{\red}{{\operatorname{red}}}
\newcommand{\reg}{{\operatorname{reg}}}

\newcommand{\Sing}{\operatorname{Sing}}

\newcommand{\Spec}{\operatorname{Spec}}

\renewcommand{\to}[1][]{\xrightarrow{\ #1\ }}

\makeatletter
\newcommand*{\da@rightarrow}{\mathchar"0\hexnumber@\symAMSa 4B }
\newcommand*{\da@leftarrow}{\mathchar"0\hexnumber@\symAMSa 4C }
\newcommand*{\xdashrightarrow}[2][]{%
  \mathrel{%
    \mathpalette{\da@xarrow{#1}{#2}{}\da@rightarrow{\,}{}}{}%
  }%
}
\newcommand{\xdashleftarrow}[2][]{%
  \mathrel{%
    \mathpalette{\da@xarrow{#1}{#2}\da@leftarrow{}{}{\,}}{}%
  }%
}
\newcommand*{\da@xarrow}[7]{%
  \sbox0{$\ifx#7\scriptstyle\scriptscriptstyle\else\scriptstyle\fi#5#1#6\m@th$}%
  \sbox2{$\ifx#7\scriptstyle\scriptscriptstyle\else\scriptstyle\fi#5#2#6\m@th$}%
  \sbox4{$#7\dabar@\m@th$}%
  \dimen@=\wd0 %
  \ifdim\wd2 >\dimen@
    \dimen@=\wd2 %
  \fi
  \count@=2 %
  \def\da@bars{\dabar@\dabar@}%
  \@whiledim\count@\wd4<\dimen@\do{%
    \advance\count@\@ne
    \expandafter\def\expandafter\da@bars\expandafter{%
      \da@bars
      \dabar@ 
    }%
  }%
  \mathrel{#3}%
  \mathrel{%
    \mathop{\da@bars}\limits
    \ifx\\#1\\%
    \else
      _{\copy0}%
    \fi
    \ifx\\#2\\%
    \else
      ^{\copy2}%
    \fi
  }%
  \mathrel{#4}%
}
\makeatother

\newtheoremstyle{citing}%
  {}%
  {}%
  {\itshape}%
  {}%
  {\bfseries}%
  {\textbf{.}}%
  {.5em}%
  {\thmnote{#3}}%

\theoremstyle{definition}

\newtheorem{remark}[thm]{Remark}
{\theoremstyle{citing}
\newtheorem*{custom}{}}

\makeatletter
\newsavebox\myboxA
\newsavebox\myboxB
\newlength\mylenA

\newcommand*\xtilde[2][0.8]{%
    \sbox{\myboxA}{$\m@th#2$}%
    \setbox\myboxB\null%
    \ht\myboxB=\ht\myboxA%
    \dp\myboxB=\dp\myboxA%
    \wd\myboxB=#1\wd\myboxA%
    \sbox\myboxB{$\m@th\widetilde{\copy\myboxB}$}%
    \setlength\mylenA{\the\wd\myboxA}%
    \addtolength\mylenA{-\the\wd\myboxB}%
    \ifdim\wd\myboxB<\wd\myboxA%
       \rlap{\hskip 0.5\mylenA\usebox\myboxB}{\usebox\myboxA}%
    \else
        \hskip -0.5\mylenA\rlap{\usebox\myboxA}{\hskip 0.5\mylenA\usebox\myboxB}%
    \fi}

\newbox\usefulbox

\def\getslant #1{\strip@pt\fontdimen1 #1}

\def\xxtilde #1{\mathchoice
 {{\setbox\usefulbox=\hbox{$\m@th\displaystyle #1$}%
    \dimen@ \getslant\the\textfont\symletters \ht\usefulbox
    \divide\dimen@ \tw@ 
    \kern\dimen@ 
    \xtilde{\kern-\dimen@ \box\usefulbox\kern\dimen@ }\kern-\dimen@ }}
 {{\setbox\usefulbox=\hbox{$\m@th\textstyle #1$}%
    \dimen@ \getslant\the\textfont\symletters \ht\usefulbox
    \divide\dimen@ \tw@ 
    \kern\dimen@ 
    \xtilde{\kern-\dimen@ \box\usefulbox\kern\dimen@ }\kern-\dimen@ }}
 {{\setbox\usefulbox=\hbox{$\m@th\scriptstyle #1$}%
    \dimen@ \getslant\the\scriptfont\symletters \ht\usefulbox
    \divide\dimen@ \tw@ 
    \kern\dimen@ 
    \xtilde{\kern-\dimen@ \box\usefulbox\kern\dimen@ }\kern-\dimen@ }}
 {{\setbox\usefulbox=\hbox{$\m@th\scriptscriptstyle #1$}%
    \dimen@ \getslant\the\scriptscriptfont\symletters \ht\usefulbox
    \divide\dimen@ \tw@ 
    \kern\dimen@ 
    \xtilde{\kern-\dimen@ \box\usefulbox\kern\dimen@ }\kern-\dimen@ }}%
 {}}

\newcommand*\xoverline[2][0.75]{%
    \sbox{\myboxA}{$\m@th#2$}%
    \setbox\myboxB\null%
    \ht\myboxB=\ht\myboxA%
    \dp\myboxB=\dp\myboxA%
    \wd\myboxB=#1\wd\myboxA%
    \sbox\myboxB{$\m@th\overline{\copy\myboxB}$}%
    \setlength\mylenA{\the\wd\myboxA}%
    \addtolength\mylenA{-\the\wd\myboxB}%
    \ifdim\wd\myboxB<\wd\myboxA%
       \rlap{\hskip 0.5\mylenA\usebox\myboxB}{\usebox\myboxA}%
    \else
        \hskip -0.5\mylenA\rlap{\usebox\myboxA}{\hskip 0.5\mylenA\usebox\myboxB}%
    \fi}

\def\xxoverline #1{\mathchoice
 {{\setbox\usefulbox=\hbox{$\m@th\displaystyle #1$}%
    \dimen@ \getslant\the\textfont\symletters \ht\usefulbox
    \divide\dimen@ \tw@ 
    \kern\dimen@ 
    \overline{\kern-\dimen@ \box\usefulbox\kern\dimen@ }\kern-\dimen@ }}
 {{\setbox\usefulbox=\hbox{$\m@th\textstyle #1$}%
    \dimen@ \getslant\the\textfont\symletters \ht\usefulbox
    \divide\dimen@ \tw@ 
    \kern\dimen@ 
    \xoverline{\kern-\dimen@ \box\usefulbox\kern\dimen@ }\kern-\dimen@ }}
 {{\setbox\usefulbox=\hbox{$\m@th\scriptstyle #1$}%
    \dimen@ \getslant\the\scriptfont\symletters \ht\usefulbox
    \divide\dimen@ \tw@ 
    \kern\dimen@ 
    \xoverline{\kern-\dimen@ \box\usefulbox\kern\dimen@ }\kern-\dimen@ }}
 {{\setbox\usefulbox=\hbox{$\m@th\scriptscriptstyle #1$}%
    \dimen@ \getslant\the\scriptscriptfont\symletters \ht\usefulbox
    \divide\dimen@ \tw@ 
    \kern\dimen@ 
    \xoverline{\kern-\dimen@ \box\usefulbox\kern\dimen@ }\kern-\dimen@ }}%
 {}}
\makeatother

\clearpage{}%

\newcommand{\lex}{{\operatorname{lex}}}

\DeclareMathOperator{\Sol}{Sol}

\DeclareMathOperator{\LM}{LM}
\DeclareMathOperator{\init}{in}

\makeatletter
\@namedef{subjclassname@2020}{
	\textup{2020} Mathematics Subject Classification}
\makeatother

\subjclass[2020]{34M15, 33C70 (primary), 34M35, 13P10, 14Q15 (secondary).}
\keywords{Algebraic Analysis, hypergeometric function, characteristic variety, singular locus, holonomic function}

\begin{document}
\title[Algebraic Analysis of ${_1F_{\!\!\;1}}$ of a Matrix Argument]{Algebraic Analysis of the Hypergeometric\\Function $\,{_1F_{\!\!\;1}}\,$ of a Matrix Argument}

\author[P. G\"{o}rlach]{Paul G\"{o}rlach}
\address{Paul G\"{o}rlach\\ Fakult\"at f\"ur Mathematik\\ Technische Universit\"at Chemnitz\\Reichen\-hainer Stra\-\ss e 39, 09126 Chemnitz, Germany}
\email{paul.goerlach@mathematik.tu-chemnitz.de}

\author[C. Lehn]{Christian Lehn}
\address{Christian Lehn\\ Fakult\"at f\"ur Mathematik\\ Technische Universit\"at Chemnitz\\Reichenhainer Stra\ss e 39, 09126 Chemnitz, Germany}
\email{christian.lehn@mathematik.tu-chemnitz.de}

\author[A.-L. Sattelberger]{Anna-Laura Sattelberger}
\address{Anna-Laura Sattelberger\\Max-Planck-Institut f\"{u}r Mathematik in den Naturwissenschaften\\Inselstra{\ss}e 22, 04103 Leipzig, Germany}
\email{anna-laura.sattelberger@mis.mpg.de}

\begin{abstract}
In this article, we investigate Muirhead's classical system of differential operators for the hypergeometric function $\,{_1F_{\!\!\;1}}\,$ of a matrix argument. We formulate a conjecture for the combinatorial structure of the characteristic variety of its Weyl closure which is both supported by computational evidence as well as theoretical considerations. In particular, we determine the singular locus of this system.
\end{abstract}
\setcounter{tocdepth}{1}
\maketitle
\tableofcontents

\section{Introduction}
Hypergeometric functions are probably the most famous special functions in mathematics and their study dates back to Euler, Pfaff, and Gau{\ss}, earlier contributions to the development of the theory are due to Wallis, Newton, and Stirling, we refer to~\cite{Dut84}. 
Around the origin, they have the series expansion
\begin{equation}\label{eq hypergeometric introduction}
\,{}_{p}F_{\!q}(a_{1},\ldots ,a_{p};c_{1},\ldots ,c_{q})(x)\, \coloneqq\, \sum_{n=0}^{\infty }\;{\frac {(a_{1})_{n}\cdots (a_{p})_{n}}{(c_{1})_{n}\cdots (c_{q})_{n}}}\ {\frac {x^{n}}{n!}},
\end{equation}
where $p,\, q$ are non-negative integers with $\,q+1\geq p\,$ and \mbox{$\,(a)_n=a\cdots (a+n-1)\,$} denotes the Pochhammer symbol. Hypergeometric functions are ubiquitous in mathematics and physics: they are intimately related to the theory of differential equations and show up at prominent places in physics such as the hydrogen atom. In recent years, there has been renewed interest in the subject coming from the connection with toric geometry established in~\cite{GKZ90,GZK89} and the interplay with mirror symmetry, see also the article~\cite{RSSW20} in this volume for more details and further references.

A natural generalization are hypergeometric functions of a matrix argument $X$ as introduced by Herz in~\cite[Section 2]{Her55} using the Laplace transform. Herz was building on work of Bochner~\cite{Boc52}. Ever since, they have been a recurrent topic in the theory of special functions. In~\cite[Section 5]{Con63}, Constantine expressed these functions as a series of zonal polynomials, thereby establishing a link with the representation theory of $\GL_n$. This series expansion bears a striking likeness to~\eqref{eq hypergeometric introduction} and is usually written as
\begin{equation}\label{eq definition hypergeometric function matrix introduction}
{_pF_{\!q}}(a_1,\ldots,a_p;c_1,\ldots,c_q)(X) \, \coloneqq \,  \sum_{n=0}^{\infty} \sum_{\lambda \, \vdash n} \frac{(a_1)_{\lambda}\cdots (a_p)_{\lambda}}{(c_1)_{\lambda}\cdots(c_q)_{\lambda}} \frac{C_{\lambda}(X)}{n!},
\end{equation}
where the $\,\lambda\,$ are partitions of $\,n\,$ and the $(a_i)_\lambda$, $(c_j)_\lambda$ are certain generalized Pochhammer symbols, see Definition~\ref{definition hypergeometric function}. 

In this article, we examine the differential equations the hypergeometric function $\,{_1F_{\!\!\;1}}(a;c)\,$ of a matrix argument $\,X\,$ satisfies from the point of view of algebraic analysis. If $\,X\,$ is an $(m\times m)$-matrix, the function~\eqref{eq definition hypergeometric function matrix introduction} only depends on the eigenvalues $\,x_1,\ldots,x_m\,$ counted with multiplicities. So we may equally well assume that $\,X=\diag(x_1,\ldots,x_m)\,$ is a diagonal matrix. In~\cite{Mui70}, Muirhead showed that the linear partial differential operators 
\begin{equation}\label{eq definition diff ops}
g_k \, \coloneqq \, x_k\partial_k^2 \,+\, (c-x_k)\partial_k \,+\,\frac{1}{2} \left( \sum_{\ell\neq k} \frac{x_\ell}{x_k-x_\ell}(\partial_k - \partial_\ell)\right) \,-\,a,
\end{equation}
$k=1,\ldots,m$, annihilate $\,{_1F_{\!\!\;1}}(a;c)\,$ wherever they are defined. We denote by $\,P_k\,$ the differential operator obtained from $\,g_k\,$ by clearing denominators and consider the left ideal $\,I_m\coloneqq( P_1,\ldots,P_m ) \,$ in the Weyl algebra $D_m$, see Section~\ref{section annihilators}. We refer to $\,I_m\,$ as the \emph{Muirhead ideal} or the \emph{Muirhead system} of differential equations and denote by $W(I_m)\,$ its Weyl closure. Our main result is: 
\begin{custom}[Theorem~\ref{thm:singularLocus}]
  The singular locus of $\,I_m\,$ agrees with the singular locus of $W(I_m)$. It is the hyperplane arrangement
  \begin{equation}\label{eq definition diagonal and coordinate hyperplanes}
  \, \slocus \, \coloneqq \, \left\{ x \in \mathbb{C}^m \ \middle| \ \prod_{k=1}^mx_k \prod_{\ell\neq k} (x_k - x_\ell) = 0  \right\}.
  \end{equation}
\end{custom}
 This leads to a lower bound for the characteristic variety of $I_m$, by which we essentially mean the characteristic variety of the $D_m$-module $D_m/I_m$. We would like to point out that the terminology used in this article is a slight modification and refinement of the usual definition in the theory of $D$-modules, taking  scheme-theoretic structures into account. For details, see Definition~\ref{definition holonomic} and the remarks thereafter.

\begin{custom}[Corollary~\ref{corollary lower bound char im}]
  The characteristic variety of $\,W(I_m)\,$ contains the zero section and the conormal bundles of the irreducible components of $\slocus$, i.e.,
 \begin{equation}\label{eq lower bound}
\begin{aligned}
\ch(W(I_m)) \,\supseteq\, &V\left(\xi_1, \ldots, \xi_m\right) \, \cup \, \bigcup_i V(x_i, \xi_1, \ldots, \widehat{\xi_i}, \ldots, \xi_m) \\ &\ \ \cup \bigcup_{i \neq j} V(x_i - x_j, \,\xi_i + \xi_j,\, \xi_1, \ldots, \widehat{\xi_i}, \ldots, \widehat{\xi_j}, \ldots, \xi_m).
\end{aligned}
\end{equation}
\end{custom}

Here, $\widehat{(\cdot)}$ means that the corresponding entry gets deleted. Note that the varieties on the right hand side of~\eqref{eq lower bound} are conormal varieties for the natural symplectic structure on $T^*\A^m$, see Section~\ref{section conormality}. More precisely, they are the conormal varieties to the irreducible components of the divisor $\,\slocus\,$ of singularities of the Muirhead system. To formulate our conjecture about the structure of the characteristic variety of $W(I_m)$, we introduce the following notation. Let $\,J_0|J_1\ldots J_k\,$ denote a partition of $[m]=\{1,\ldots,m\}$, such that only $J_0$ may possibly be empty. We denote by $Z_{J_0|J_1\ldots J_k}$ the linear subspace given by the vanishing of all $\,x_i\,$ for $\,i\in J_0\,$ and all $\,x_i-x_j\,$ for $\,i,j \in J_\ell$ and $\ell\in[k]$. For a smooth subvariety $Y\subseteq \A^m$, we denote by $\, N^*Y\subseteq T^*\A^m \,$ the conormal variety to $Y$. Then our conjecture can be phrased as follows:

\begin{custom}[Conjecture~\ref{conjecture charVar}]
Let $C_{J_0|J_1\ldots J_k} \coloneqq N^*Z_{J_0|J_1\ldots J_k}$. The (reduced) characteristic variety of $\, W( I_m )\, $ is the following arrangement of $m$-dimensional linear spaces:
\[
\ch(W(I_m))^{\text{\rm red}} \,= \, \bigcup_{[m] \,= \, J_0 \sqcup \dots \sqcup J_k} C_{J_0|J_1\ldots J_k}.
\]
In particular, it has $\,B_{m+1}\,$ many irreducible components, where $\,B_n\,$ denotes the $n$-th Bell number.
\end{custom}

By an explicit analysis of the differential operators in $ I_m$, we also obtain an upper bound for $\ch(I_m)$. For a partition $J_0|J_1 \dots J_k$, we define certain subspaces $\widehat C_{J_0|J_1 \dots J_k} \subseteq T^*\A^m\,$ such that $\,C_{J_0|J_1 \dots J_k} \subseteq \widehat C_{J_0|J_1 \dots J_k}\,$ with equality if and only if $\,\abs{J_\ell}\leq 2\,$ for all $\ell \geq 1$, see~\eqref{eq subspace chat} for the precise definition.

\begin{custom}[Proposition~\ref{proposition upper bound}]
The (reduced) characteristic variety of $\, I_m \,$ is contained in the arrangement of the linear spaces $\widehat C_{J_0|J_1 \dots J_k}$:
 \[
\ch(I_m)^{\text{\rm red}} \, \subseteq \, \bigcup_{[m] \,= \, J_0 \sqcup J_1 \sqcup \dots \sqcup J_k} \widehat C_{J_0|J_1\ldots J_k}.
\]
\end{custom}

It is the upper and lower bound together with explicit computations in the computer algebra system {\tt Singular} for small values of $m$, see Section~\ref{section small m}, that led us to formulate Conjecture~\ref{conjecture charVar}. We believe that it may contribute to a better understanding of the hypergeometric function ${_1F_{\!\!\;1}}$. As $\, I_m\,$ turns out to be non-holonomic in general, it seems that one should rather work with its Weyl closure $W(I_m)$, for which, in general, generators are not known. Clearly, one has $\ch(W(I_m))\subseteq \ch(I_m)$. Therefore, Proposition~\ref{proposition upper bound} in particular also gives an upper bound for $\ch(W(I_m))$.

\subsection*{Applications and related work.}
Hypergeometric functions of a matrix argument possess a rich structure and are highly fascinating objects. Not surprisingly, there is by now a long list of interesting applications in various areas such as number theory, numerical mathematics, random matrix theory, representation theory, statistics, and others; the following short list does not claim to be exhaustive.

The relation to representation theory and statistics is classical. For the link to representation theory, we refer to~\cite{BO93} and references therein. The connection with multivariate statistics was already present in~\cite{Her55} through the connection to the Wishart distribution, see~\cite[Section~8]{Her55}.

Unlike in the one-variable case, hypergeometric functions of a matrix argument have been studied from the point of view of holonomic systems only recently. The first instance we know of appeared in arithmetic~\cite{IKO12}. Motivated by the study of Siegel modular forms and the computation of special values of $L$-functions, the authors of~\cite{IKO12} study solutions of certain systems of differential equations. They are equivalent to Muirhead's system, see e.g.\ their Proposition~7.4 and Theorem~7.5. Holonomicity is shown explicitly in~\cite[Theorem 9.1]{IKO12}. Apart from number theory, hypergeometric functions of a matrix argument and holonomic systems also made an appearance in random matrix theory~\cite{DL15}.

A large impetus came from numerical analysis with the advent of the holonomic gradient descent and the holonomic gradient method developed in~\cite{NNNOSTT11}. These methods allowed to numerically evaluate and minimize several functions that are of importance in multivariate statistics. In~\cite{NNNOSTT11} and~\cite{KNNT14}, these methods are applied to the Fisher--Bingham distribution. In~\cite{HNTT13}, the holonomic gradient method is used to approximate the cumulative distribution function of the largest root of a Wishart matrix. Motivated by this method, several teams, mainly in Japan, have studied Muirhead's systems from the $D$-module point of view such as~\cite{HNTT13,HTT18,Noro,SSTOT13}.
This is the starting point for our contribution. We examine the $D$-module theoretic properties of Muirhead's ideal for the hypergeometric function $\,{_1F_{\!\!\;1}}\,$ of a matrix argument from a completely and consistently algebraic point of view.

\subsection*{Outline.}
This article is organized as follows. In Section~\ref{section tools}, we recall some basic facts about the Weyl algebra and $D_m$-ideals. We recall the notion of  holonomic functions and give a characterization that is well suited for testing holonomicity. In Section~\ref{section hypergeometric}, we discuss hypergeometric functions of a matrix argument. 
In Section~\ref{section annihilators}, we define the Muirhead ideal $\, I_m\,$ and collect what is known about holonomicity of $\, I_m\,$ and its Weyl closure. 
Section~\ref{section solutions muirhead} contains our main results. We investigate the Muirhead ideal of operators annihilating $\,{_1F_{\!\!\;1}}\,$ and determine its singular locus. This section also contains some results about holomorphic and formal solutions of the Muirhead system. 
The characteristic variety of this ideal and its Weyl closure is investigated in Section~\ref{section charvar}.
Conjecture~\ref{conjecture charVar} suggests that the characteristic variety of the Weyl closure can be described in a combinatorial way, using partitions of sets. We also discuss some basic computations in low dimensions.

For computations around the characteristic variety, we mainly used the  libraries {\tt dmod}~\cite{Singular-dmod}, {\tt dmodapp}~\cite{Singular-dmodapp},  and {\tt dmodloc}~\cite{Singular-dmodloc} in {\tt Singular}~\cite{Singular}. We also performed some Gr\"{o}bner basis computations in the rational Weyl algebra, where we used the {\tt Mathematica}~\cite{Mathematica} package {\tt HolonomicFunctions}~\cite{HolFun}.

\subsection*{Acknowledgments.}
We are thankful to Andr{\'a}s L{\H o}rincz, Christian Sevenheck, Bernd Sturmfels, and Nobuki Takayama for insightful discussions. We are grateful to the anonymous referee for valuable hints on literature and for proposing a strategy that led to an alternative proof of our main theorem using different techniques and enabled us to remove a technical condition on a parameter. We refer to the discussion in \cref{section solutions muirhead} and \cref{sec:appendix} for details.

P.G.\ acknowledges partial support by the DFG grant Se 1114/5-2. C.L.\ was supported by the DFG through the research grants Le 3093/2-2 and  Le 3093/3-1.

\section{The Weyl algebra}\label{section tools}
In this section, we recall basic facts about the Weyl algebra, the characteristic variety, and the definition of holonomic functions. We mainly follow the presentation and notation given in~\cite{SST00,SatStu}.

\subsection{Ideals and characteristic varieties}\label{section weyl}
We start by introducing some notation and terminology. Throughout this article, $\N$ denotes the natural numbers including $0$. For $m\in \N_{>0}$, we denote by 
$$D_m \,  \coloneqq \,  \mathbb{C}[x_1,\ldots,x_m]\langle \partial_1,\ldots,\partial_m \rangle$$
the {\em $m$-th Weyl algebra} and by 
$$R_m  \,  \coloneqq \,   \mathbb{C}\left( x_1,\ldots,x_m \right) \langle \partial_1,\ldots,\partial_m \rangle$$
the ring of differential operators with rational functions as coefficients. 
In this article, we refer to $\,R_m\,$ as {\em $m$-th rational Weyl algebra}. For a commutative ring $A$, we will abbreviate $\, A[x] =  A[x_1,\ldots,x_m]\,$ the polynomial ring and $\,A(x)=   A(x_1,\ldots,x_m)\, $ the field of rational functions. We will also use $\,\xi\,$ as a set of variables so that e.g. $\mathbb{C}(x)[\xi]  =\mathbb{C}(x_1,\ldots,x_m)[\xi_1,\ldots,\xi_m]$.

For a vector $\,w = (u,v) \in \R^{2m}\,$ with $\,u+v \geq 0\,$ component-wise, we define a partial order on the monomials $\,x^\alpha \di^\beta \in \C[x_1,\ldots,x_m] \langle \partial_1,\ldots,\partial_m \rangle\,$ for $\,\alpha,\beta \in \N^{m}\,$ by comparing the quantity 
$$\wdeg(x^\alpha \di^\beta ) \, \coloneqq \, \alpha\cdot u + \beta \cdot v \,=\, \sum_{i=1}^m \alpha_iu_i + \beta_iv_i,$$
where the indices refer to the coordinates of the vectors. We refer to $\,w\,$ as a \emph{weight vector} and to $\,\wdeg\,$ as the \emph{$w$-degree}. With the notation $\,e=(1,\ldots,1)\in \N^m\,$ and $\,w=(0,e)\,$ we recover the \emph{order} of a partial differential operator as the leading exponent for this $w$-degree. 

Given an operator $\,P\in D_m\,$ and a weight vector $w\in \R^{2m}$, we define its \emph{initial form} $\,\init_w(P)\,$ to be the sum of all terms of maximal $w$-degree. Note that one has to write $\,P\,$ in the basis $\,x^\alpha \di^\beta\,$ in order to compute the $w$-degree, i.e., one has to bring all differentials to the right.

The initial form $\,\init_w(P)\,$ can be viewed as the class of $\,P\,$ of the associated graded algebra $\,\gr_w(D_m)\,$ to the filtration of $\,D_m\,$ induced by $w$. The relation $\,\partial_i x_i - x_i \partial_i = 1$ in $\,D_m\,$ induces the relation
  \[\partial_i x_i - x_i \partial_i \,=\, 
  \begin{cases} 0 &\text{if } \,u_i + v_i > 0\\
                1 &\text{if } \,u_i + v_i = 0
  \end{cases}
  \qquad \qquad \text{in }\gr_{(u,v)}(D_m).
  \]
To highlight this commutator relation notationally, one writes $\,\xi_i\,$ instead of $\,\partial_i\,$ in $\,\gr_{(u,v)}(D_m)\,$ for all indices $\,i\,$ with $u_i + v_i = 0$. In particular,
  \[\gr_{(u,v)}(D_m) \,=\, \C[x][\xi] \ \text{ if } \,u+v > 0 \quad \text{ and }    \quad
    \gr_{(u,v)}(D_m) \,=\, D_m \ \text{ if } \,u+v = 0.\]

A {\em $D_m$-ideal} is a \textit{left} $D_m$-ideal. For a $D_m$-ideal $I$, the \emph{initial ideal} with respect to $\,w\,$ is the left ideal
\begin{equation}\label{eq initial ideal}
\init_{w}(I) \, \coloneqq \, \left( \, \init_w(P) \, \middle| \, P\in I\, \right) \,\subseteq\, \gr_w(D_m).
\end{equation}
A {\em $D_m$-module} is a \textit{left} $D_m$-module. 
$\text{Mod}(D_m)$ denotes the category of $D_m$-modules. Likewise for $R_m$-ideals and $R_m$-modules, respectively.
Next we recall the important notions of a characteristic variety and of holonomicity. 

\begin{definition}\label{definition holonomic}
The \emph{characteristic variety} of a $D_m$-ideal $\,I\,$ is the subscheme of~$\,\A^{2m}\,$
determined by the ideal $\,  \init_{(0,e)}(I)  \subseteq  \mathbb{C}[x_1,\ldots,x_m][\xi_1,\ldots,\xi_m]\,  $ 
and is denoted by $\ch(I)$. 
The $D_m$-ideal $\,I\,$ is called \textit{holonomic} if $\,\init_{(0,e)}(I)\,$ has dimension $m$.
\end{definition}
\pagebreak
\begin{remark}\label{remark charvar}\
\begin{enumerate}
	\item Note that $\left( 0 \right) $ and $\,D_m\,$ are not holonomic. 
Therefore, if $\,I\,$ is a holonomic ideal, it is a non-zero, proper $D_m$-ideal. 
\item Recall that as a consequence of an important theorem of Sato--Kawai--Kashiwara~\cite{SKK73}, we have $\dim Z\geq m$ for all irreducible components $\,Z\,$ of $\ch(I)$, see also the discussion in Section~\ref{section conormality}.
\item It is worthwhile to remark that the scheme structure of the characteristic variety 
is not uniquely determined by the $D_m$-module $D_m/I$.
Intrinsic invariants of $\,D_m/I\,$ are the set $\,\ch(I)^\red\,$ and the multiplicity of its irreducible components, 
see e.g.~\cite[Section~2.2]{HTT08}. The point is that---unlike in the commutative world---$I$ cannot be recovered as the annihilator of the $D_m$-module $\,D_m/I$, and so there can be $\,I\neq J \subseteq D_m\,$ with $\,D_m/I \isom D_m/J\,$.
\end{enumerate}
\end{remark}

\subsection{Conormality of the characteristic variety}\label{section conormality}
We remark that $\, \A^{2m} = \Spec\C[x_1,\ldots,x_m,\xi_1,\ldots,\xi_m]\,$ should actually be considered as the cotangent bundle $\,T^*\A^m\,$ where the $\,\xi_i\,$ are the coordinates in the fiber of the canonical morphism $\, T^*\A^m \to \A^m\,$ and the $\,x_i\,$ are the coordinates in the base. Being a cotangent bundle, $T^*\A^m$ carries a natural (algebraic) symplectic form $\,\sigma\,$ which can explicitly be described in coordinates as
\[
\sigma \,= \, dx_1 \wedge d\xi_1 + \cdots + dx_m \wedge d\xi_m.
\]
The symplectic structure gives rise to the notion of a \emph{Lagrangian subvariety}, that is, a subvariety $\,Z\subseteq T^*\A^m\,$ such that at every smooth point $z\in Z^\reg$, the tangent space $\,T_zZ \subseteq T_z(T^*\A^m) \,=\, T^*\A^m\,$ is isotropic (i.e., $\sigma$ vanishes identically on this subspace) and maximal with this property. Note that a Lagrangian subvariety automatically has dimension $m$. Examples for Lagrangian subvarieties in $\,T^*\A^m\,$ are conormal varieties. Given a subvariety $X \subseteq \A^m$, the associated {\em conormal variety} $\,N^*_X\,$ is defined as the Zariski closure of the conormal bundle $N_{X^\reg/\A^m}^* \subseteq T^*\A^m$. This is always a Lagrangian subvariety. We will make use of the following (special case of) important results due to Sato--Kawai--Kashiwara~\cite[Theorem 5.3.2]{SKK73}, see also Gabber's article~\cite[Theorem I]{Gab81} for an algebraic proof. 

\begin{theorem}\label{theorem conormal}
Let $\,I\,$ be a $D_m$-ideal. Then $\,\ch(I) \subseteq  T^*\A^m\,$ is coisotropic. If $\,I\,$ is holonomic, every irreducible component $\,Z\,$ of the characteristic variety $\, \ch(I) \,$ is a conormal variety. In particular, $Z$ is Lagrangian.
\end{theorem}

To be more precise, the references above show that $\,Z\,$ is Lagrangian. By definition, the characteristic variety is stable under the $\C^\ast$-action given by  scalar multiplication in the fibers of $ T^*\A^m \to \A^m$, and therefore it is conormal by \cite[Lemma~(3.2)]{Kas7475}, see also~\cite[Theorem~E.3.6]{HTT08}.

\subsection{Holonomic functions}\label{section holonomic}
In this section, we recall the definition of a holonomic function and give a characterization of this notion which turns out to be very useful in practice.

\begin{definition}\label{definition holonomic function}
Let $\,M\,$ be a $D_m$-module and $f\in M$.  The {\em annihilator of $\,f\,$} is the $D_m$-ideal
$$\text{Ann}_{D_m}\left(f\right)\, \coloneqq \, \left\{ P\in D_m \mid P \bullet f =0 \right\}.$$ 
An element $f\in M\,$ is {\em holonomic} if its annihilator is a holonomic $D_m$-ideal.
\end{definition}

The definition generalizes in an obvious way to arbitrary subsets $N\subseteq M$.
If $\,M\,$ is a space of functions (e.g. holomorphic, multivalued holomorphic, smooth etc.) and $\,f\in M\,$ is holonomic, then we refer to $\,f\,$ as a {\em holonomic function}. The definition of a holonomic function first appeared in the article~\cite{Zei90} of Zeilberger.

\begin{definition}\label{definition weyl closure}
  The {\em Weyl closure}  of a $D_m$-ideal $\,I\,$ is the $D_m$-ideal 
  $$W(I)\,  \coloneqq \,  \left( R_mI\right)  \,  \cap  \,  D_m .$$ 
  We clearly have $ I \subseteq  W(I)$. A $D_m$-ideal  $\,I\,$ is {\em Weyl closed} if $\, I = W(I)\, $ holds.
\end{definition}

In general, it is a challenging task to compute the Weyl closure of a \mbox{$D_m$-ideal}, see~\cite{Tsa00} for the one-dimensional case and~\cite{Tsa02} in general.
The following property is in particular shared by spaces of functions.

\begin{definition}
  A $D_m$-module $\,M\,$ is {\em torsion-free} if it is torsion-free as module over $\mathbb{C}[x_1,\ldots,x_m]$.
\end{definition}

This class of $D_m$-modules allows to deduce further properties of annihilating $D_m$-ideals.
\begin{lemma}\label{lemma annihilator weyl closed}
Let $\,M\in \text{Mod}\left(D_m\right)\,$ be torsion-free and $N$ a subset of $M$. 
Then $\,\text{Ann}_{D_m}\left(N\right)\,$ is \textit{Weyl closed}.
\end{lemma}
\begin{proof}
Write a given $\,P\in W( \text{Ann}_{D_m}(N))\,$ as $\,P=\sum_i q_iP_i\,$ where $\,q_i \in R_m\,$ and $P_i \in \text{Ann}_{D_m}(N)$. We choose $\,h\in \mathbb{C}[x_1,\ldots,x_m]\,$ such that $h P \in \text{Ann}_{D_m}(N)$. Then for every $\,f\in N\,$ we have
$\,hP\bullet f=0\,$ and therefore $P\bullet f=0$, since $\,M\,$ is torsion-free.
\end{proof}

\begin{definition}\label{definition holonomic rank}
 For a $D_m$-ideal $I$, 
its {\em singular locus} is the set
   \begin{align}
   \Sing(I) \,\coloneqq \, \bigcup_{Z \, \subseteq \, \ch(I)} \overline{\pi(Z)} \,\subseteq\, \A^m,
   \end{align}
 where $\,\pi\,$ denotes the projection $\,T^* \A^m \to \A^m\,$ and the union is over all irreducible components $\,Z\,$ of $\,\ch(I)$\, distinct from the zero section $\,\{\xi_1=\cdots=\xi_m=0\}\,$ as sets. Moreover,
 we denote by
 \begin{align}
 \hrk \left( I \right)  \,  \coloneqq \,   \dim_{\mathbb{C}(x)}\left( \mathbb{C}(x)[\xi]/\mathbb{C}(x)[\xi]\init_{(0,e)}(I) \right)  
  \,  = \,   \dim_{\mathbb{C}(x)} \left( R_m/ R_mI\right)
\end{align}	
the \em holonomic rank of $I$.
\end{definition}

The second equality is a standard fact, we refer to~\cite[Section 1.4]{SST00}. If $\,I\,$ is a holonomic $D_m$-ideal, $\hrk(I)$ gives the dimension of the space of holomorphic solutions to $\,I\,$ in a simply connected domain outside the singular locus of $\,I\,$ by the theorem of Cauchy--Kowalevski--Kashiwara Theorem~\cite[p.~44]{Kas83}, see also~\cite[Theorem 1.4.19]{SST00}. The following result clarifies the relationship between the holonomic rank and holonomicity.

\begin{lemma}[\cite{SST00}, Theorem~1.4.15]\label{lemma weyl closure holonomic}
	Let $\,I\,$ be a $D_m$-ideal. If $\,I\,$ has finite holonomic rank, then its Weyl closure $\,W(I)\,$ is a holonomic $D_m$-ideal.
\end{lemma}

The following characterization of holonomicity is useful.

\begin{prop}\label{proposition equivalent definitions holonomic}
Let $\,M\,$ be a torsion-free $D_m$-module and $f\in M$. Then the following statements are equivalent.
\begin{enumerate}[label=(\arabic*)]
	\item\label{item one proposition equivalent definitions holonomic} $f$ is holonomic.
	\item\label{item two proposition equivalent definitions holonomic} For all $ k= 1,\ldots,m$, there exists a natural number $\,m(k)\in \N \,$ 
	and a non-zero differential operator 
	$P_k=\sum_{\ell=0}^{m(k)} a_\ell(x_1,\ldots,x_m)\partial_k^\ell \in  \Ann_{D_m}(f).$
	\item\label{item three proposition equivalent definitions holonomic} The annihilator of $\,f\,$ has finite holonomic rank.
\end{enumerate}
\end{prop}
\begin{proof}
By the elimination property for holonomic ideals in the Weyl algebra (cf. \cite[Lemma~4.1]{Zei90}, with a proof attributed to   Bernstein), \ref*{item one proposition equivalent definitions holonomic}$\Rightarrow$\ref*{item two proposition equivalent definitions holonomic} holds. The equivalence \ref*{item two proposition equivalent definitions holonomic}$\iff$\ref*{item three proposition equivalent definitions holonomic} is obvious.
Finally, \ref*{item three proposition equivalent definitions holonomic}$\Rightarrow$\ref*{item one proposition equivalent definitions holonomic} follows from combining Lemma~\ref{lemma annihilator weyl closed} with Lemma~\ref{lemma weyl closure holonomic}.
\end{proof}

Without the condition of torsion-freeness, there are counterexamples to the validity of \ref*{item three proposition equivalent definitions holonomic}$\Rightarrow$\ref*{item one proposition equivalent definitions holonomic}, see e.g.~\cite[Example 1.4.10]{SST00}.

\section{Hypergeometric functions of a matrix argument}\label{section hypergeometric}
In this section, we are going to introduce the hypergeometric functions of a matrix argument in the sense of Herz~\cite{Her55}, see Definition~\ref{definition hypergeometric function}. We will follow Constantine's approach~\cite{Con63} via zonal polynomials.

\subsection{Zonal polynomials}\label{section zonal polynomials}
Zonal polynomials are important in multivariate analysis with applications in multivariate statistics. Their theory has been developed by James in~\cite{Jam60,Jam61} and subsequent works, see the introduction of Chapter 12 of Farrell's monograph~\cite{Far76} for a more complete list. The definition given by James in~\cite{Jam61} relies on representation theoretic work of \'{E}.~Cartan~\cite{Car29} and James also credits Hua~\cite{Hua55,Hua59}, see~\cite{Hua79} for an English translation. As a general reference, the reader may consult the monographs of Farrell~\cite[Chapter 12]{Far76}, Takemura~\cite{Tak84}, and Muirhead~\cite{Mui82}. The presentation here follows~\cite[Chapter~7]{Mui82}.

Let $\,m\,$ be a fixed positive integer. Throughout, we only consider partitions of the form $\, \lambda\,= \,(\lambda_1,\ldots,\lambda_m)\,$ of an integer $\, d = \abs{\lambda}\coloneqq \lambda_1+\cdots+\lambda_m\,$ with $\, \lambda_1 \geq \lambda_2 \geq\cdots \geq\lambda_m \geq 0\,$ if not explicitly stated otherwise.
\begin{definition}\label{def:zonal}
For all partitions $ \, \lambda = (\lambda_1,\ldots,\lambda_m)\, $ of $d$, 
the {\em zonal polynomials} $\, C_\lambda \, \in\, \C[x_1,\ldots,x_m]\,$ are defined to be the unique symmetric 
homogeneous polynomials of degree $\,d\,$ satisfying the following three properties. 
\begin{enumerate}
	\item The leading monomial with respect to the lexicographic  order $\,\prec_\lex\,$ with $\,x_m \prec_\lex \dots \prec_\lex x_1\,$ is 
  \mbox{$\LM_{\prec_\lex}(C_\lambda) =x^\lambda = x_1^{\lambda_1}\cdots x_m^{\lambda_m}$.}
	\item The functions $\,C_\lambda\,$ are eigenfunctions of the operator 
	$$	\Delta \,= \, \sum_{i=1}^m x_i^2 \del_i^2 \,+\, \sum_{\substack{i,j = 1\\ i\neq j}}^m \frac{x_i^2}{x_i-x_j} \del_i ,$$
	 i.e, $\Delta \bullet C_\lambda =\alpha_\lambda \cdot C_\lambda$ for some $ \alpha_\lambda \in \C$.
	\item We have $$(x_1 + \cdots + x_m)^d  \,= \, \sum_{\abs{\lambda}=d} C_\lambda.$$
\end{enumerate}
\end{definition}

The uniqueness and existence of course have to be proven, we refer to~\cite[Section~7.2]{Mui82}, 
where also the eigenvalues $\,\alpha_\lambda\,$ are determined to be
\begin{align*}
\alpha_\lambda \,= \, \rho_\lambda \,+\, d\cdot(m-1) \quad \textrm{with }\,\, \rho_\lambda \,= \, \sum_{i=1}^m \lambda_i(\lambda_i - i).
\end{align*} 
Zonal polynomials can be explicitly calculated by a recursive formula for the coefficients 
in a basis of monomial symmetric functions. 
From this it follows that zonal polynomials have in fact rational coefficients. 
The space of symmetric polynomials has a basis given by symmetrizations of monomials. 
We can enumerate this basis by ordered partitions; the partition of a given basis element 
is its leading exponent in the lexicographic order. 
For a partition $\, \lambda = (\lambda_1,\ldots,\lambda_m)\,$ we put:
\begin{align*}
M_\lambda\, \coloneqq  \,x^\lambda + \textrm{ all permutations } \,= \, \sum_{\mu \in \mathfrak{S}_m.\lambda} x^\mu,
\end{align*}
where $\,\mathfrak{S}_m.\lambda\,$ denotes the orbit of the $m$-th symmetric group $\mathfrak{S}_m$.
We write the zonal polynomials with respect to this basis:
\begin{align*}
C_\lambda \,= \, \sum_{\mu \leq \lambda} c_{\lambda,\mu} M_\mu.
\end{align*}
Zonal polynomials can now be computed explicitly thanks to the following recursive formula:
\begin{align*}
c_{\lambda,\mu} \,= \, \mathlarger{\sum}_{\kappa} \ \frac{\kappa_i-\kappa_j}{\rho_\lambda - \rho_\mu}\ c_{\lambda,\kappa},
\end{align*}
where the sum runs over all (not necessarily ordered) partitions $\, {\kappa\,= \, (\kappa_1,\ldots,\kappa_m)}\,$ 
such that there exist $\,i < j\,$ with $\, \kappa_k = \mu_k\,$ for all $\,k\neq i,j\,$ and 
$ \kappa_i = \mu_i+t$, $\kappa_j = \mu_j-t$ for 
some $\,t \in \{1,\ldots,\mu_j\}\,$ and such that $\, \mu < \kappa \leq \lambda\,$ after reordering $\kappa$.

\subsection{Hypergeometric functions of a matrix argument}\label{section definition hypergeometric}
Let $\,X \in \C^{m\times m}\,$ be a square matrix and $\, \lambda = (\lambda_1,\ldots,\lambda_m)\,$ a partition. One defines the zonal polynomial $\,C_\lambda(X)\,$ as
\[
C_\lambda(X) \, \coloneqq  \, C_\lambda(x_1,\ldots,x_m),
\]
where $\,x_1,\ldots,x_m\,$ are the eigenvalues of $\,X\,$ counted with multiplicities. Note that $\,C_\lambda(X)\,$ is well-defined because $\,C_\lambda\,$ is a symmetric polynomial.

\begin{definition}\label{definition hypergeometric function}
The {\em hypergeometric function of a matrix argument} $\,X\,$ is given by
\begin{equation}\label{eq definition hypergeometric function matrix}
{_pF_{\! q}}(a_1,\ldots,a_p;c_1,\ldots,c_q)(X) \, \coloneqq \,  \sum_{k=0}^{\infty} \sum_{\lambda \, \vdash k} \frac{(a_1)_{\lambda}\cdots (a_p)_{\lambda}}{(c_1)_{\lambda}\cdots(c_q)_{\lambda}} \frac{C_{\lambda}(X)}{k!},
\end{equation}
where, for a partition $\lambda =(\lambda_1,\ldots,\lambda_{m})$, the symbol $\,(a)_{\lambda}\,$ denotes the {\em generalized Pochhammer symbol}
\begin{align*}
(a)_{\lambda} \, \coloneqq \, \prod_{i=1}^{m} \left( a-\frac{i-1}{2}\right)_{\lambda_i}.
\end{align*}
Here, for an integer $\ell$, the quantity $\, (a)_\ell =a(a+1)\cdots (a+\ell-1)\,$ with $\,(a)_0=1\,$ is the usual Pochhammer symbol.
\end{definition}

The parameters $\,a_1, \ldots, a_p\,$ and $\,c_1, \ldots, c_q\,$ in this definition are allowed to attain all complex values such that all the denominators $\,(c_i)_\lambda\,$ do not vanish. Explicitly,
\begin{equation} \label{eq:conditionOnParameters}
a_1, \ldots, a_p \in \C \ \ \text{and} \ \  c_1, \ldots, c_q \in 
\begin{cases} 
\C\setminus (-\N) &\text{if } \,m = 1, \\
\C \setminus \big\{\frac{k}{2} \mid k \in \Z, \, k \leq m-1\big\} &\text{if } \,m \geq 2.
\end{cases}
\end{equation}
\begin{rmk}	If $X=\text{diag}(x_1,0,\ldots,0)$, it follows straight forward from Definition~\ref{def:zonal} of zonal polynomials that $ {_pF_{\! q}}(a_1,\ldots,a_p;c_1,\ldots,c_q)(X) $ is the classical hypergeometric function  $ {_pF_{\! q}}(a_1,\ldots,a_p;c_1,\ldots,c_q)(x_1)$ in one variable. Therefore, Definition~\ref{definition hypergeometric function} is indeed an appropriate generalization of hypergeometric functions in one variable.
\end{rmk}

The convergence behavior of the hypergeometric function of a matrix argument is analogous to the one-variable case, basically with the same proof.
For $\,p\leq q$, this series converges for all $X$. For $p=q+1$, this series converges for $\Vert X \Vert  < 1$, where $\,\Vert  \cdot \Vert \,$ denotes the maximum of the absolute values of the eigenvalues of $X$. 
If $\,p> q+1$, the series diverges for all $X\neq 0$.

\section{Annihilating ideals of \texorpdfstring{${_1F_{\!\!\;1}}$}{1\textsubscript F\textsubscript 1}}\label{section annihilators}
Let $\,{_1F_{\!\!\;1}}\,$ be the hypergeometric function of a matrix argument as introduced in Definition~\ref{definition hypergeometric function}. In this section, we systematically study a certain ideal that annihilates ${_1F_{\!\!\;1}}$. This function depends on two complex parameters $a,c$ satisfying condition \eqref{eq:conditionOnParameters}, which in this case means
\begin{equation} \label{eq:conditionOnC}
  \begin{cases} c \notin -\N &\text{if } m=1, \\ c \notin \{\frac{k}{2} \mid k \in \Z, k \leq m-1\} &\text{if } m \geq 2. \end{cases}
\end{equation}
As discussed in the last section, the value of this function on a symmetric matrix $\,X\in\C^{m\times m}\,$ is the same as the value on the unique semisimple element in the $\,\GL_m(\C)\,$ (conjugacy) orbit closure of $X$.
We may thus restrict our attention to the case where $\,X\,$ is diagonal. Then this hypergeometric function satisfies the following differential equations.

\subsection{Setup and known results about the annihilator}\label{section setup annihilator}
\begin{thm}{\cite[Theorem 7.5.6]{Mui82}}\label{theorem muirhead}
Let $m\in \N_{>0}$ and let $a,c \in \C$ be parameters with $c$ satisfying \eqref{eq:conditionOnC}. The function $\, {_1F_{\!\!\;1}}(a;c)\,$ of a diagonal matrix argument $\, X=\diag(x_1,\ldots,x_m)\, $ 
is the unique solution $\,F\,$ of the system of the $\,m\,$ linear partial differential equations given by the operators
	\begin{align}\label{ann1F1}
	g_k \, \coloneqq \,x_k\partial_k^2 + \left( c-\frac{m-1}{2}-x_k + \frac{1}{2}\sum_{\ell\neq k} \frac{x_k}{x_k-x_\ell} \right)\partial_k 
	- \frac{1}{2}\left( \sum_{\ell\neq k} \frac{x_\ell}{x_k-x_\ell}\partial_\ell \right) - a, 
	\end{align}
	$ k= 1,\ldots,m,$
	subject to the conditions that $\,F\,$ is symmetric in $x_1,\ldots,x_m$, and $\,F\,$ is analytic at $X=0$, and $F(0)=1$.
\end{thm}

In fact, we will point out in \cref{prop:uniqueSolAroundOrigin} that in this theorem, the condition of symmetry in $\,x_1, \ldots, x_m\,$ can be dropped as it is implied by the other conditions.
By using the identity
\begin{align*}
 \frac{x_k}{x_k-x_\ell} \,= \, 1\,+\,\frac{x_\ell}{x_k-x_\ell},
 \end{align*}
the operators from~\eqref{ann1F1} can be written as 
\begin{align}\label{ann1F1re}
g_k \,= \, x_k\partial_k^2 \,+\, (c-x_k)\partial_k \,+\,\frac{1}{2} \left( \sum_{\ell\neq k} \frac{x_\ell}{x_k-x_\ell}(\partial_k - \partial_\ell)\right) \,-\,a.
\end{align}
Clearing the denominators in~\eqref{ann1F1}, we obtain 
\begin{equation}\label{eq definition p}
P_k \,  \coloneqq \,  \left( \prod_{\ell\neq k} (x_k-x_\ell)\right) \cdot g_k \,\in\, D_m, \quad k \,= \, 1,\ldots,m.
\end{equation}

\begin{definition}\label{definition im}
We denote by $\,I_m\,$ the $D_m$-ideal generated by $\, P_1,\ldots,P_m\,$ and call it the \emph{Muirhead ideal}.
\end{definition}

Note that, by construction,
\[
R_m I_m \,=\, (g_1,\ldots,g_m).
\]
Our goal is to systematically study the ideal $ I_m$. In this direction, Hashiguchi--Numata--Takayama--Takemura obtained the following result in~\cite{HNTT13}.

\begin{theorem}[{\cite[Theorem~2]{HNTT13}}]\label{theorem holonomic rank im}
For the graded lexicographic term order on $R_m$, a Gröbner basis of $\,R_m I_m\,$ is given by $\{g_k = x_k \di_k^2 + \text{l.o.t.} \mid k=1,\ldots,m\}$.
\end{theorem}

An immediate consequence is:

\begin{corollary}\label{corollary hntt}
The holonomic rank of $\, I_m\,$ is given by $\hrk(I_m)=2^m$. In particular, the Weyl closure $\,W(I_m)\,$ of $\, I_m\,$ and the function $\,{_1F_{\!\!\;1}}\,$ of a diagonal matrix are holonomic.
\end{corollary}
\begin{proof}
This immediately follows from Theorem~\ref{theorem holonomic rank im} and Lemma~\ref{lemma weyl closure holonomic}.
\end{proof}

At the end of Section~5 in~\cite{HNTT13}, it is conjectured that $\, I_m\,$ is holonomic. Via direct computation they show that $\,I_2\,$ is holonomic in Appendix A of the paper. One can still verify holonomicity of $\,I_3\,$ for generic parameters $a,c$ through a  computation in {\tt Singular}. It turns out, however, that the above conjecture does not hold. We are thankful to N.~Takayama for pointing out that the $D_4$-ideal $\, I_4 \,$ was shown to be non-holonomic in the Master's thesis~\cite{Kmaster}. We give an easy alternative argument for this in Example~\ref{example m4}.

\section{Analytic solutions to the Muirhead ideal}\label{section solutions muirhead}
In this section, we determine the singular locus of the Muirhead ideal $\, I_m\,$ and of its Weyl closure:

\begin{thm} \label{thm:singularLocus}
  Let $m \in \N _{>0}$ and let $a,c \in \C$ be parameters. Then the singular locus of $\, I_m\,$ agrees with the singular locus of $W(I_m)$. It is the hyperplane arrangement
  \begin{align}
  \slocus \, \coloneqq \, \left\{x \in \C^m \ \middle|\ \prod_{i=1}^m x_i \prod_{j\neq i} (x_i - x_j) = 0\right\}.
  \end{align}
\end{thm}

\noindent To be more precise, in this section we will prove the statement under the additional
\begin{assumption}\label{assumption parameter}
The parameter $\,c\,$ satisfies condition \eqref{eq:conditionOnC}.
\end{assumption}
Note that this condition makes the function $\,{_1F_1(a;c)}\,$ well-defined. However, we would like to point out that this assumption is not necessary; a proof of the stronger statement is given in \cref{sec:appendix}. We are grateful to the anonymous referee for suggesting to investigate {\em restriction modules} which are the central tool in the proof presented there. As these are different techniques, we deem it worthwhile to also present our original proof, which is the purpose of this section.

The inclusion $\,\Sing(I_m) \subseteq \slocus\,$ is readily seen from \[\init_{(0,e)}(P_i) \,=\, x_i \prod_{j \neq i} (x_i - x_j) \partial_i^2.\]
To prove the reverse containment, we investigate analytic solutions to the Muirhead system locally around points in the components of the arrangement $\slocus$. Our main technical tool is the following observation resembling \cite[Theorem 2.5.5]{SST00}:

\begin{lemma} \label{lem:solutionInitials}
  Let $\,I\,$ be a $D_m$-ideal and let $u \in \R_{\geq 0}^m$. Then
  \[\dim \Sol_{\C\llbracket x \rrbracket}(I) \, \leq \, \dim \Sol_{\C\llbracket x \rrbracket}(\init_{(-u,u)}(I)),\]
  where $\,\Sol_{\C\llbracket x \rrbracket}(\cdot)\,$ denotes the solution space in the formal power series ring $\C\llbracket x\rrbracket$.
\end{lemma}

\begin{proof}
  For $f = \sum_{\alpha \in \N^m} \lambda_\alpha x^\alpha \in \C\llbracket x \rrbracket$, we denote\footnote{Our notation for the initial of a formal power series differs from the one used, among others, in~\cite{SST00,SatStu}. Ours is more coherent with the definition of initial forms of linear differential operators.}
  \[\init_{-u}(f) \, \coloneqq \sum_{\substack{u^T \alpha \text{ min.} \\ \text{with } \lambda_\alpha \neq 0}} \lambda_\alpha x^{\alpha} \in \C\llbracket x \rrbracket.\]
  If $\,P = \init_{(-u,u)}(P) + \tilde{P} \in D_m\,$ annihilates $f = \init_{-u}(f) + \tilde{f}$, then
  \[0 \,=\, P \bullet f \,=\, \init_{(-u,u)}(P) \bullet \init_{-u}(f) \,+\, \tilde{P} \bullet f \,+\, \init_{(-u,u)}(P) \bullet \tilde f\]
  and all monomials appearing in the expanded expression $\,\tilde{P} \bullet f + \init_{(-u,u)}(P) \bullet f\,$ are of higher $u$-degree than those of $\init_{(-u,u)}(P) \bullet \init_{-u}(f)$. Hence, $\init_{(-u,u)}(P)$ annihilates $\init_{-u}(f)$. This shows that for every $D_m$-ideal $I$, we have
  \begin{equation} \label{eq:initialsCommuteWithSols}
 \left\{\init_{-u}(f) \mid f \in \Sol_{\C\llbracket x \rrbracket}(I)\right\}\, \subseteq \,\Sol_{\C\llbracket x \rrbracket}(\init_{(-u,u)}(I)).
  \end{equation}
  Let $\,F\,$ be a basis of the solution space $\Sol_{\C\llbracket x \rrbracket}(I)$. Replacing $\,F\,$ by a suitable linear combination of its elements, we can assure that the initial forms $\,\init_{-u}(f)\,$ for $\,f \in F\,$ are linearly independent. Then \eqref{eq:initialsCommuteWithSols} implies
  \[\dim \Sol_{\C\llbracket x \rrbracket}(\init_{(-u,u)}(I)) \,\geq\, |\{ \init_{-u}(f)\mid f\in F \}| \,=\, |F| \,=\, \dim \Sol_{\C\llbracket x \rrbracket}(I). \qedhere\]
\end{proof}

In the following two lemmata, we apply \cref{lem:solutionInitials} to the Muirhead system and bound the spaces of analytic solutions locally around general points in $\slocus$. Note that up to $\mathfrak S_m$-symmetry, there are two types of components in $\slocus$, namely $\{x \in \C^m \mid x_1 = 0\}$ and $\{x \in \C^m \mid x_1 = x_2\}$. \cref{lem:solAroundZeroComp} considers points that lie in exactly one component of $\,\slocus\,$ of the first type, while \cref{lem:solAroundDiagComp} is concerned with the second type.

\begin{lemma} \label{lem:solAroundZeroComp}
  Let $\,p \in \C^m\,$ be a point with distinct coordinates, one of which is zero. If $a,c \in \C$ with $c \notin (m-1)/2 -\N$, then the space of formal power series solutions to $\, I_m\,$ centered at $\,p\,$ is of dimension at most $2^{m-1}$.
\end{lemma}

\begin{proof}
  Since $\, I_m\,$ is invariant under the action of the symmetric group $\mathfrak{S}_m$, we may assume that the point $\,p = (p_1, \ldots, p_m)\,$ has the unique zero coordinate $p_1 = 0$. Studying formal power series solutions to $\, I_m\,$ around $\,p\,$ is equivalent to substituting $\,x_i\,$ by $\,x_i + p_i\,$ in each of the generators $\,P_1, \ldots, P_m\,$ and to studying  the solutions in $\,\C \llbracket x \rrbracket\,$ of the resulting operators. 
  Let us define $u \coloneqq  (3,2,\ldots,2) \in \R^m$.
  Examining the expression for $P_1, \ldots, P_m$, we observe that
  \begin{equation} \label{eq:shiftedInitialsZeroComp}
  \begin{array}{lcl}
  \init_{(-u,u)}\big(\restr{P_1}{x\, \mapsto x+p}\big) &= &(-1)^{m-1}
  p_2 p_3 \cdots p_m \frac{1}{x_1} \theta_1\Big(\theta_1 + c - \frac{m + 1}{2}\Big) \quad \ \text{and} \\
  \init_{(-u,u)}\big(\restr{P_i\,}{x \,\mapsto x+p}\big) &= &p_i \prod_{j \neq i} (p_i - p_j) \frac{1}{x_i^2} \theta_i(\theta_i-1) \qquad \text{for all }\, i \geq 2,
  \end{array}
  \end{equation}
  where $\theta_i \coloneqq x_i \partial_i$ and $\,\restr{P_i}{x\, \mapsto x+p}\,$ denotes the operator obtained from $\,P_i\,$ by replacing $x$ with $x+p$. Note that an operator $\,P(\theta_1,\ldots,\theta_m)\in\C[\theta_1, \ldots, \theta_m] \subseteq D_m\,$ acts on the one-dimensional vector spaces $\,\C \cdot x^\alpha\,$ for $\,\alpha \in \N^m\,$ with eigenvalue $P(\alpha)$. In particular, 
  the space of solutions in $\,\C\llbracket x \rrbracket\,$ of the operators \eqref{eq:shiftedInitialsZeroComp} is spanned by the $\,2^{m-1}\,$ monomials $\,x^{\alpha}\,$ with $\,\alpha_1 = 0\,$ and $\,\alpha_i \in \{0,1\}\,$ for all $i \geq 2$. Here, we have used that $\,\frac{m + 1}{2} - c \notin \N _{>0}\,$ by \cref{assumption parameter} on $c$, which guarantees that formal power series solutions to $\,\theta_1 + c - \frac{m + 1}{2}\,$ are constant in $x_1$.
  In particular, from \cref{lem:solutionInitials}, we conclude
  \begin{align*}
  \dim \Sol_{\C\llbracket x \rrbracket}\big(\restr{I_m}{x \, \mapsto x+p}\big) & \,\leq \, \dim \Sol_{\C\llbracket x \rrbracket}\big(\!\init_{(-u,u)}\big(\restr{I_m}{x \, \mapsto x+p}\big)\big) \\ &\, \leq \, \dim \Sol_{\C\llbracket x \rrbracket}\big(\!\init_{(-u,u)}\big(\restr{P_i}{x \,\mapsto x+p}\big) \mid i = 1, \ldots, m\big) \,=\, 2^{m-1}.
  \end{align*}
\end{proof}

\begin{lemma} \label{lem:solAroundDiagComp}
  Let $\,p = (p_1, \ldots, p_m) \in (\C^*)^m\,$ with $\,\#\{p_1, \ldots, p_m\} = m-1$. For all $a,c \in \C$, the space of formal power series solutions to $\, I_m\,$ centered at $\,p\,$ is of dimension at most $2^{m-2} \cdot 3$.
\end{lemma}

\begin{proof}
  We proceed similar to the proof of \cref{lem:solAroundZeroComp}. By symmetry of $I_m$, we may assume that $p_1 = p_2$, while all other pairs of coordinates of $\,p\,$ are distinct. Denote $e := (1, \ldots, 1) \in \N^m$. Then 
  $$ \begin{array}{lcl}
  \init_{(-e,e)}\big(\restr{P_1}{x \,\mapsto x+p}\big)  & \,=\, &\phantom{-}\frac{1}{2}p_1 \prod_{j=3}^m (p_1-p_j) \cdot  (2(x_1-x_2)\partial_1^2+\partial_1-\partial_2), \\
  \init_{(-e,e)}\big(\restr{P_2}{x \,\mapsto x+p}\big)  &\,=\, &-\frac{1}{2}p_2 \prod_{j=3}^m (p_2-p_j) \cdot  (2(x_1-x_2)\partial_2^2+\partial_1-\partial_2), \\
  \init_{(-e,e)}\big(\restr{P_i\,}{x \,\mapsto x+p}\big)& \,=\, &\phantom{-\frac{1}{2}} p_i \, \prod_{j\neq i} \; (p_i-p_j) \cdot \frac{1}{x_j^2} \theta_i(\theta_i-1) \qquad  \text{for }\,i \geq 3\\
  \end{array}$$
  with $\theta_i \coloneqq x_i \partial_i$. From the identity $\,\theta_i \bullet x^\alpha = \alpha_i x^\alpha\,$  for all $\,\alpha \in \N^m\,$ we deduce that a basis of $\,\Sol_{\C\llbracket x \rrbracket}\big(\big\{\restr{P_i}{x\,\mapsto x+p} \mid i\big\}\big)\,$ is given by $f(x_1,x_2) x_3^{\alpha_3} x_4^{\alpha_4} \cdots x_m^{\alpha_m}$, where $\,{\alpha_3, \ldots, \alpha_m \in \{0,1\}}\,$ and where $\,f\,$ varies over a basis of
  \[\Sol_{\C\llbracket x_1, x_2 \rrbracket}\Big(2(x_1-x_2)\partial_1^2+\partial_1-\partial_2,\; 2(x_1-x_2)\partial_2^2+\partial_1-\partial_2\Big).\]
  The latter is a $3$-dimensional vector space spanned by $\{1,\, x_1+x_2,\, x_1^2+6x_1 x_2 + x_2^2\}$. This can be easily verified as follows.
  After the change of variables 
  \[y_1 \, \coloneqq \,  (x_1+x_2)/2, \quad y_2 \,\coloneqq \, (x_1-x_2)/2, \quad \partial_{y_1} \,=\, \partial_1 + \partial_2, \quad \partial_{y_2} \,=\, \partial_1 - \partial_2,\]
  this system becomes
  \begin{align*}
  \left(y_2\left(\partial_{y_1} + \partial_{y_2}\right)^2 + \partial_{y_2}\right) \bullet f \,=\, 0, \qquad 
  \left(y_2\left(\partial_{y_1} - \partial_{y_2}\right)^2 + \partial_{y_2}\right) \bullet f \,=\, 0
  \end{align*}
  From summing these two equations, we observe that a solution $f \in \C\llbracket y_1, y_2 \rrbracket$ needs to be annihilated by the operator $\partial_{y_1} \partial_{y_2}$. Therefore, we can write any solution as $f = \sum_{i \geq 0} \lambda_i y_1^i + \sum_{j \geq 1} \mu_j y_2^j$. Plugging this into $(y_2(\partial_{y_1} + \partial_{y_2})^2 + \partial_{y_2}) \bullet f = 0$, we observe that $\,\lambda_i = \mu_i = 0\,$ for all $i \geq 3$, $\mu_1 = 0$ and $\lambda_2 = -2 \mu_2$, leading to the basis of solutions $$\left\{1,\, 2y_1 = x_1 + x_2, \, 8y_1^2-4y_2^2 = x_1^2+6x_1x_2+x_2^2\right\}.$$
  With this, we have argued that the solution space of $\,\init_{(-e,e)}\big(\restr{I_m}{x\,\mapsto x+p}\big)\,$ is at most $3 \cdot 2^{m-2}$-dimensional. Together with~\cref{lem:solutionInitials}, this proves the claim.
\end{proof}

\begin{proof}[Proof of \cref{thm:singularLocus}]
  First, we observe that 
  \[\init_{(0,e)}(P_i) \,=\, x_i \prod_{j \neq i} (x_i - x_j) \xi_i^2\]
  and hence 
  \[\ch(I_m)^{\text{\rm red}} \,\subseteq \,\bigcap_{i=1}^m \Big(V(\xi_i) \cup V(x_i) \cup \bigcup_{j\neq i} V(x_i - x_j)\Big)\, \subseteq \,\pi^{-1}(\slocus) \cup V(\xi_1, \ldots, \xi_m),\]
  where $\,\pi \colon T^* \A^m \to \A^m\,$ denotes the natural projection. By definition of the singular locus, this proves the containment
  \[\Sing(W(I_m)) \, \subseteq \, \Sing(I_m)\, \subseteq\, \slocus.\]
  
  For the reverse inclusion, consider a point $\,p \in \C^m\,$ contained in exactly one irreducible component of $\slocus$.  By \cref{lem:solAroundZeroComp} and \cref{lem:solAroundDiagComp}, the space of formal power series solutions to $\, I_m\,$ (or, equivalently, to $W(I_m)$) around $\,p\,$ is of dimension strictly smaller than $2^m = \hrk(I_m) = \hrk(W(I_m))$. In particular, $p$ needs to be a singular point of $\, I_m\,$ and of $W(I_m)$, as otherwise the Cauchy--Kowalevski--Kashiwara Theorem implies the existence of $\,2^m\,$ linearly independent analytic solutions around~$p$. In particular, the singular loci of $\, I_m\,$ and of $\,W(I_m)\,$ must contain those points. Since singular loci are closed, we conclude that they contain the entire arrangement $\slocus$. 
\end{proof}

  \begin{remark}\label{remark parameters}
    The condition \eqref{eq:conditionOnC} on the parameter $\,c\,$ is very natural from the point of view of analytic functions, as the hypergeometric function $\,{_1F_{\!\!\;1}}(a;c)\,$ of a diagonal matrix argument is only defined under this condition.\footnote{Note however that our proof of \cref{thm:singularLocus} in this section relies only on the condition that $c \notin (m-1)/2 - \N$, which is slightly weaker than \eqref{eq:conditionOnC}.} 
    However, the Muirhead ideal itself is defined for arbitrary $\,a,c \in \C\,$ and is the more interesting object from the point of view of $D$-module theory.
  \end{remark}

The description of the singular locus in \cref{thm:singularLocus} gives rise to the following lower bound on the characteristic variety. In \cref{section charvar}, we will also discuss an upper bound and a conjectural description of the characteristic variety.

\begin{cor}\label{corollary lower bound char im}
  The characteristic variety of $\,W(I_m)\,$ contains the zero section and the conormal bundles of the irreducible components of $\slocus$, i.e.,
  \begin{align*}
  \ch(W(I_m)) \,\supseteq\, &V\left(\xi_1, \ldots, \xi_m\right) \, \cup \, \bigcup_i V(x_i, \xi_1, \ldots, \widehat{\xi_i}, \ldots, \xi_m) \\ &\quad \cup \bigcup_{i \neq j} V(x_i - x_j, \,\xi_i + \xi_j,\, \xi_1, \ldots, \widehat{\xi_i}, \ldots, \widehat{\xi_j}, \ldots, \xi_m).
  \end{align*}
\end{cor}

\begin{proof}
  As already noted in the introduction after \eqref{eq lower bound}, the linear spaces on the right hand side of the claimed inclusion are conormal varieties. By Theorem~\ref{theorem conormal}, the conormal varieties to the irreducible components of $\,\Sing(W(I_m))\,$ are contained in $\ch(W(I_m))$. Moreover, the zero section $\,V\left(\xi_1, \ldots, \xi_m\right)\,$ is always contained in the characteristic variety. Theorem~\ref{thm:singularLocus} concludes the proof.
\end{proof}

Above, we have studied bounds on solutions to the Muirhead system locally around points in $\,\C^m\,$ contained in exactly one component of $\slocus$, while the Cauchy--Kowalevski--Kashiwara Theorem describes the behavior around points in $\C^m \setminus \slocus$. A more detailed study around special points $\,p \in \slocus\,$ where several components of $\,\slocus\,$ intersect may be of interest.

We finish this section by looking at the most degenerate case: $p = 0$. Recall from \cref{theorem muirhead} that $\,{}_1F_{\!\!\;1}\,$ is the unique analytic solution to $\, I_m\,$ around $\,0\,$ that is symmetric and normalized to attain the value $\,1\,$ at the origin. In fact, the restricting factor assuring uniqueness here is not the symmetry, but the analyticity around~$0$. Namely, using the techniques presented before, we arrive at the following refinement of \cref{theorem muirhead}:

\begin{prop} \label{prop:uniqueSolAroundOrigin}
  Let $m \in \N _{>0}$ and let $a,c \in \C$ be parameters with $c$ satisfying~\eqref{eq:conditionOnC}. Then $\,{}_1F_{\!\!\;1}(a;c)\,$ is the unique formal power series solution to $\, I_m\,$ around $\,0\,$ with ${}_1F_{\!\!\;1}(a;c)(0) = 1$. In particular, ${}_1F_{\!\!\;1}(a;c)$ is the unique convergent power series solution to $\,I_m\,$ around $\,0\,$ with \mbox{${}_1F_{\!\!\;1}(a;c)(0) = 1$}.
\end{prop}

\begin{proof}
  Consider any weight vector $\,u \in \R_{\geq 0}^m\,$ with $0 < u_1 < u_2 < \cdots < u_m$. From the definition of $P_1, \ldots, P_m$, we see that for all $i \in \{1,\ldots,m\}$:
  \[
  \init_{(-u,u)}(P_i) \,=\, \frac{(-1)^{i-1}}{2} x_1 \cdots x_{i-1} \cdot x_i^{m-i-1} \cdot \left(2\theta_i^2 + (2c-i-1)\theta_i-\sum_{j=i+1}^m \theta_j\right),
  \]
  where $\theta_i \coloneqq  x_i \partial_i$.
  In particular, the Weyl closure of $\,\init_{(-u,u)}(I)\,$ contains the operators 
  $Q_i \coloneqq  2\theta_i^2 + (2c-i-1)\theta_i-\sum_{j=i+1}^m \theta_j$. 
  The action of operators in $\,\C[\theta_1, \ldots, \theta_m] \subseteq D_m\,$ on $\,\C\llbracket x \rrbracket\,$ diagonalizes with respect to the basis of $\,\C\llbracket x \rrbracket\,$ given by the monomials. 
  In particular, $\Sol_{\C\llbracket x \rrbracket}(Q_1, \ldots, Q_m)$ is a subspace of $\,\Sol_{\C\llbracket x \rrbracket}(\init_{(-u,u)}(I))\,$ spanned by monomials. Therefore, by \cref{lem:solutionInitials}, it suffices to show that the only monomial annihilated by $\,Q_1, \ldots, Q_m\,$ is $1$.
  
  Let $\,\alpha \in \N^m\,$ be such that $\,x^\alpha\,$ is annihilated by $Q_1, \ldots, Q_m$. Assume for contradiction that $\,\alpha \neq 0\,$ and let $\,i \in \{1,\ldots,m\}\,$ be maximal such that $\alpha_i \neq 0$. Then 
  \[0 \,=\, Q_i \bullet x^\alpha \,=\, 2\alpha_i^2 + (2c-i-1)\alpha_i-\sum_{j=i+1}^m \alpha_j \,=\, \alpha_i\cdot(2\alpha_i + 2c-i-1).\]
  Note that $\,c \notin \{\frac{k}{2} \mid k \in \Z,\, k \leq m-1\}\,$ guarantees $\,2\ell + 2c-i-1 \neq 0\,$ for all positive integers~$\ell$. This contradicts the assumption $\alpha_i \neq 0$. We conclude that 
  \[\Sol_{\C\llbracket x \rrbracket}(\init_{(-u,u)}(I)) \, \subseteq \, \Sol_{\C\llbracket x \rrbracket}(Q_1, \ldots, Q_m) \,=\, \C \cdot \{1\}\]
  and therefore $\dim \Sol_{\C\llbracket x \rrbracket}(I_m) \leq 1$.
  The last claim is now immediate.
\end{proof}

\section{Characteristic variety of the Muirhead ideal}\label{section charvar}
In this section, we give a conjectural description of the (reduced) characteristic variety of the Weyl closure of the Muirhead ideal $I_m $, see Conjecture~\ref{conjecture charVar}. The conjecture based on our computations and further evidence is provided by the partial results obtained in Corollary~\ref{corollary lower bound char im} and Proposition~\ref{proposition upper bound}. The description of~$\, \ch\left(W(I_m)\right)\,$ is combinatorial in nature and would imply that the number of irreducible components is given by the $(m+1)$-st Bell number $B_{m+1}$.

\subsection{Conjectural structure of the characteristic variety}
Let us first explain some notations. 

\begin{notation}\label{notation partitions}
We denote $[m]=\{1,\ldots,m\}$. We consider partitions of this set $[m]=J_0 \sqcup J_1 \sqcup \ldots \sqcup J_k$, where $\,J_0\,$ is allowed to be empty, the $\,J_i\,$ with $\,i\neq 0\,$ are nonempty, and we consider the $\,J_1, \ldots, J_k\,$ as unordered. Taking into account that $\, J_0\,$ plays a distinguished role, we denote such a partition by $J_0 \mid J_1\ldots J_k$.
\end{notation}

For a partition $ [m] = J_0 \mid J_1  \dots J_k$, we denote by $\,C_{J_0|J_1\ldots J_k}\,$ the $m$-dimensional linear subspace
\begin{align}\label{eq subspace c}
V\: \bigg( \{ x_j \mid j \in J_0 \} \cup \Big\{\sum_{i \in J_\ell} \xi_i \bigmid \ell \,= \, 1,\ldots,k\Big\}
 \cup \bigcup_{\ell=1}^k \{x_i-x_j \mid i,j \in J_\ell\}\bigg)
\end{align}
of $T^*\A^m = \A^{2m} = \Spec \C[x_1,\ldots,x_m,\xi_1, \ldots,\xi_m]$.

Let $\,B_k \in \N \,$ denote the $k$-th \emph{Bell number}, i.e., the number of partitions of a set of size $k$. For example $B_1=1$, $B_2=2$, $B_3=5$, $B_4=15$, $B_5=52$, and so on. 
For the Muirhead ideal $\, I_m$, the characteristic variety of its Weyl closure $\,W(I_m)\,$ has the following conjectural description.

\begin{conj} \label{conjecture charVar}
 The (reduced) characteristic variety of $\,W(I_m )\,$ is the following arrangement of $m$-dimensional linear spaces:
 \[
\ch(W(I_m))^{\text{\rm red}} \,= \, \bigcup_{[m] \,= \, J_0 \sqcup \dots \sqcup J_k} C_{J_0|J_1\ldots J_k}.
\]
In particular, $\ch(W(I_m))$ has $\,B_{m+1}\,$ many irreducible components.
\end{conj}

As $\,I_4\,$ is not holonomic, it does not seem reasonable to make predictions about $\ch(I_m)$. The better object to study is its Weyl closure, which is challenging to compute. The appearance of the Bell numbers in the conjecture is explained by the following observation:

We have a bijection of sets
\begin{align}\label{partitions}
\begin{split}
&\bigcup_{k=1}^m \, \big\{ \text{Ordered partitions } \{ 0,1,\ldots,m\}  \,= \, \tilde{J}_1 \sqcup \ldots \sqcup \tilde{J}_k \big\} / \mathfrak{S}_k \\ 
\stackrel{1:1}{\rightleftarrows} \,\,  &\bigcup_{k=0}^m \, \big\{ \text{Ordered partitions } \{ 1,\ldots ,m\}  
\,= \, J_0 \sqcup J_1 \sqcup \ldots \sqcup J_k \big\} / \mathfrak{S}_k ,
\end{split}
\end{align} 
defined by $\,J_0 \coloneqq \tilde{J}_i \setminus \{0\} \, $ for $ 0\in \tilde{J}_i$, 
where on the right hand side of~\eqref{partitions}, the symmetric group $\,\mathfrak{S}_k\,$ acts on $J_1\sqcup \ldots \sqcup J_k$. It is important to note that $\,J_0\,$ is allowed to be empty, and $\, J_0\,$ is the only set among the $\,J_i\,$ and $\,\tilde{J}_j\,$ with this property. 

\subsection{Bounds for the characteristic variety}
Next, we give an upper bound for the reduced characteristic variety $\,\ch(I_m)^{\rm red}\,$ and hence a fortiori an upper bound for $\ch(W(I_m))^\red$. 
By \emph{upper bound}, we mean a variety containing the given variety. 
Note that we already proved a \emph{lower bound} for $\,\ch(W(I_m))\,$ in \cref{corollary lower bound char im}.

For a partition $\,J_0 \mid J_1\ldots J_k\,$ of $[m]$, we defined the linear subspace 
$\, C_{J_0|J_1\dots J_k}\,$ of $\A^{2m}\,$ in~\eqref{eq subspace c}. 
We denote by $\, \widehat C_{J_0|J_1\dots J_k} \, $  the linear space

\begin{align}\label{eq subspace chat}
V\bigg(\{x_j \mid j \in J_0\} \cup \bigcup_{\ell=1}^k \{x_i-x_j \mid i,j \in J_\ell\} 
\cup \Big\{\sum_{i \in J_\ell} \xi_i \mid \ell = 1,\ldots,k \text{ s.t.\ } |J_\ell| \leq 2\Big\}\bigg) 
\end{align}
of $\A^{2m}$.
Clearly, $\widehat C_{J_0|J_1\dots J_k}  \supseteq C_{J_0|J_1\dots J_k}$, 
with equality if and only if $ \,|J_\ell| \leq 2\,$ for \linebreak\mbox{$\ell = 1, \ldots, k$}.
Further evidence for Conjecture~\ref{conjecture charVar} is given by the following result.

\begin{prop}\label{proposition upper bound}
The (reduced) characteristic variety of $\, I_m \,$ is contained in the arrangement of the linear spaces $\widehat C_{J_0|J_1 \dots J_k}$:
  \[\ch(I_m)^{\text{\rm red}} \, \subseteq \, \bigcup_{[m] \,= \, J_0 \sqcup J_1 \sqcup \dots \sqcup J_k} \widehat C_{J_0|J_1 \ldots J_k}.\]
In particular, this also gives an upper bound for $\ch(W(I_m))^{\text{\rm red}}$.
\end{prop}
\begin{proof}
  The characteristic variety of $\,I_m\,$ is defined by the vanishing of the symbols $\,\init_{(0,e)}(P) \in \C[x][\xi]\,$ of all operators $P \in I_m$. 
  Hence, describing explicit symbols in $\,\init_{(0,e)}(I_m)\,$ bounds $\,\ch(I_m)\,$ from above. We observe that 
    \begin{align*} 
    \init_{(0,e)}(P_i) \,= \, x_i \cdot \left( \prod_{j \neq i} (x_i - x_j)\right) \cdot \xi_i^2 \qquad \text{for } \,i \,= \,1,\ldots,m.
    \end{align*}
  Moreover, for $i \neq j$, consider the following operators in $I_m$:
    \[S_{ij} \, \coloneqq \, x_j\cdot \left( \prod_{k \neq i,j} (x_j-x_k) \right) \cdot \partial_j^2 \cdot P_i \,+\, x_i \cdot \left( \prod_{k \neq i,j} (x_i-x_k) \right) \cdot \partial_i^2 \cdot P_j.\]
  This expression can be seen as the $S$-pair of the operators $P_i$ and $P_j$ for graded term orders on $R_m$. A straightforward computation by hand reveals that 
   \begin{align*}
     \init_{(0,e)}(S_{ij})
     \,=\, -\frac{1}{2} x_i x_j \Big( \prod_{k \neq i,j} (x_i-x_k)(x_j-x_k)\Big)
     \big(\xi_i+\xi_j\big)^3 + (x_i-x_j) Q_{ij}
   \end{align*}
    for some $Q_{ij} \in \C[x][\xi]$.
  
  Since these operators lie in the Muirhead ideal, we have
  \[
	\ch\left(I_m\right) \, \subseteq \, V\,\left(\init_{(0,e)}(P_i),\, \init_{(0,e)}(S_{ij}) \mid  i\neq j\right) \, \eqqcolon \, Z,
	\]
  so it suffices to see that $\,Z\,$ is set-theoretically contained in the union of all 
  $\,\widehat C_{J_0|J_1\dots J_k}$.
  We prove this by the comparing their fibers over $\A^m = \Spec \C[x_1,\ldots,x_m]$.
  Let $\, z = (z_1,\ldots,z_m) \in \A^m\,$ and let  
  $\, [m]= J_0 \sqcup J_1 \sqcup \dots \sqcup J_k\,$ 
  be a partition of $\,[m]\,$ such that
\begin{align}\label{special partition}
z_i =0 \, \iff \,i \in J_0 \qquad \text{and} \qquad z_i = z_j \,\iff\, \exists \ell: i,j \in J_\ell.
\end{align}
  Note that this partition is uniquely determined by the point $\,z\,$ up to permuting $J_1, \ldots, J_k$. 
  Let $\,F\,$ denote the fiber of $\,Z\,$ over the point $z$. We claim that $\,F\,$ is set-theoretically contained in the fiber of 
  $\, \widehat C_{J_0|J_1 \dots J_k}\,$ over $z$. 
  
  To prove this claim, it suffices to see that for all singletons $\,J_{\ell} = \{n\}\,$ and two-element sets $\,J_{\ell'}= \{i,j\}\,$ in our partition,
  where $1\leq \ell,\,\ell'\leq k$, the polynomials $\, \xi_n^2\,$ and $\, (\xi_i+\xi_j)^3\,$ vanish on $F$. But for those $n,i,j,$ the polynomial
\begin{equation} \label{eq:symbolOfGenerator}
\restr{\init_{(0,e)}(P_n)\,}{\mathbf{x}\,=\,z} \,=\, z_n \cdot \left( \prod_{j \neq n} (z_n - z_j)\right) \cdot \xi_n^2
\end{equation}
  is a non-zero multiple of $\,\xi_n^2\,$ by \eqref{special partition},  since $\,J_{\ell}\,$ is a singleton, and
\begin{equation} \label{eq:symbolOfSPair}
\restr{\init_{(0,e)}(S_{ij})\,}{\mathbf{x}=z} \,= \, -\frac{1}{2} z_i z_j \cdot \prod_{p \neq i,j} (z_i-z_p)(z_j-z_p) (\xi_i+\xi_j)^3
\end{equation}
is a non-zero multiple of $(\xi_i+\xi_j)^3$. Here, we have used that $\,z_i=z_j\,$ by construction of the partition $J_0|J_1 \dots J_k$.

Both \eqref{eq:symbolOfGenerator} and \eqref{eq:symbolOfSPair} vanish on $\,F\,$ by the definition of $Z$, and hence $\,\xi_n\,$ and $\,\xi_i+\xi_j\,$ vanish on the set $F^{\text{\rm red}}$, disregarding the scheme structure. 
This shows that $F^{\rm red}\subseteq  \widehat C_{J_0|J_1 \dots J_k}$. In particular, 
 \[\ch(I)^{\rm red} \, \subseteq \, Z^{\rm red} \, \subseteq \, \bigcup_{[m] \,= \, J_0 \sqcup J_1 \sqcup \dots \sqcup J_k} \widehat C_{J_0|J_1\ldots J_k},\]
 concluding the proof.
\end{proof}

\subsection{Examples}\label{section small m}
The computational difficulty of questions concerning the characteristic variety $\ch(I_m)$, the Weyl closure $W(I_m)$, its characteristic variety, irreducible components, and more increases rapidly with the number of variables $m$. For $m=2,3,$ we succeed with straightforward computations in {\tt Singular} to obtain the characteristic variety and its decomposition into irreducible components. For $m=2$, also the Weyl closure $\,W(I_m)\,$ is computable, but already for $m=3$ this is no longer feasible. For $m=4$, none of the computer calculations terminate. We provide more precise information in the following examples.

\begin{example}\label{example m2}
We consider the case $m=2$. We perform our computations for generic $a,c$, i.e., in 
\[
\mathbb{Q}(a,c)[x_1,\ldots,x_m]\langle \partial_1,\ldots,\partial_m \rangle
\]
with indeterminates $a,c$. Computations in {\tt Singular} show that the characteristic variety $\,\ch\left( I_2\right)\,$ set-theoretically decomposes into the following five irreducible components 
\begin{equation}\label{eq charvar m2}
\begin{aligned} 
 V\left(x_1,x_2\right)\, \cup \,V\left( x_1,\xi_2 \right)\,\cup\, V\left(\xi_1,x_2 \right)   
\,\cup\, V\left(\xi_1,\xi_2\right)\, \cup \, V \left(\xi_1+\xi_2,\,x_1-x_2 \right).
\end{aligned}
\end{equation}
Already for $m=2$, the ideal $\, I_m\,$ and its Weyl closure $\,W(I_m)\,$ differ. The operator
\[ 
P\, =\, g_1-g_2\,=\,(x_1\partial_1^2-x_2\partial_2^2)-(x_1\partial_1 - x_2 \partial_2) + (c-\frac{1}{2}) (\partial_1-\partial_2)
\]
is clearly in $W(I_2)\setminus I_2$. In fact, $W(I_2)=I_2+(P)$. Moreover, $ \ch(I_2)^\red =\ch(W(I_2))^\red\,$ but the multiplicities of the irreducible components are different. In the order of appearance in~\eqref{eq charvar m2}, the irreducible components have multiplicities $\,4,2,2,4,3\,$ in $\,I_2\,$ and $\,3,2,2,4,1\,$ in $W(I_2)$.

The decomposition~\eqref{eq charvar m2} will also turn out to be a byproduct of our more general result presented in Proposition~\ref{proposition upper bound}.
\end{example}

\begin{example}\label{example m3}
Next we consider the case $m=3$. Computations for generic $a,c$ in {\tt Singular}  show that $\,\ch\left( I_3\right)\,$ 
decomposes into the $\,15=B_4\,$ irreducible components
\begin{align*} 
&V(x_1,x_2,x_3) \ \cup \  V(\xi_1, x_2, x_3) \ \cup \  V(x_1, \xi_2, x_3) \ \cup \  V(x_1, x_2, \xi_3) \\
&\ \ \cup \  V(\xi_1, \xi_2, x_3) \ \cup \  V(\xi_1, x_2, \xi_3) \ \cup \  V(x_1, \xi_2, \xi_3) \ \cup \  V(\xi_1, \xi_2, \xi_3) \\
&\ \ \cup \  V(x_1-x_2,\, \xi_1+\xi_2,\, x_3)  \ \cup \  V(x_1-x_3,\, \xi_1+\xi_3,\, x_2)  \, \cup \  V(x_2-x_3, \, \xi_2+\xi_3, \,x_1) \\
&\ \ \cup \  V(x_1-x_2,\, \xi_1+\xi_2,\, \xi_3)  \ \cup \  V(x_1-x_3,\, \xi_1+\xi_3,\, \xi_2)  \ \cup \  V(x_2-x_3, \, \xi_2+\xi_3, \,\xi_1) \\
&\  \ \cup \  V(x_1 - x_2,\, x_1 - x_3,\, \xi_1+\xi_2+\xi_3).
\end{align*}

If we compare this to our upper bound for the characteristic variety $\,\ch(W(I_3))\,$ from Proposition~\ref{proposition upper bound}, we see that the only difference between the components in~\eqref{eq subspace c} and~\eqref{eq subspace chat} is that instead of $\, V(x_1,x_2,x_3)\,$ and $V(x_1-x_2,\, x_2-x_3,\,\xi_1+\xi_2+\xi_3)$, we only have the component $\, B\coloneqq V(x_1-x_2,\, x_2-x_3)\subseteq T^*\A^3\,$ in the upper bound. 
However, the Weyl closure is holonomic by Lemma~\ref{lemma weyl closure holonomic} and thus the components of its characteristic variety are the conormals to their projections to~$\,\A^3\,$ by Theorem~\ref{theorem conormal}. Such a projection is a closed subvariety of the diagonal $V(x_1-x_2,\,x_2-x_3) \subseteq \A^3$, hence either equal to it or equal to a point. The corresponding conormal varieties are $\, V(x_1-x_2,\,x_2-x_3,\,\xi_1+\xi_2+\xi_3)\,$ and the cotangent spaces to the points $\,p_\lambda\coloneqq (\lambda,\lambda,\lambda)\,$ for some $\lambda \in \C$. It turns out that the components $\, V(x_1-x_2,\,x_2-x_3,\,\xi_1+\xi_2+\xi_3)\,$ and $\,V(x_1,x_2,x_3)\,$ of $\,\ch(W(I_3))\,$ are the only ones contained in $B$. In other words, the cotangent spaces to $\,p_\lambda\,$ are not contained in the characteristic variety unless $\lambda=0$. It does not seem to be very pleasant to verify this last claim by hand. The operator $\,P\,$ of lowest order we found in $\,D_3\,$ whose symbol $\,\init_{(0,e)}(P)\,$ does not vanish on $\,p_\lambda\,$ with $\,\lambda\neq 0\,$ has order $\,4\,$ and one needs coefficients of order $\,6\,$ to show that $P\in I_3$.

It is striking that the components of $\,\ch(W(I_3))\,$ contained in $\,B\,$ are exactly those conormal bundles contained in $\,B\,$ that are bihomogeneous in the $x_i$ and the~$\xi_j$. According to Conjecture \ref{conjecture charVar}, all components should have this property but for the time being we do not see how to deduce bihomogeneity in general, see also Problem~\ref{problem group action}.
\end{example}

\begin{example}\label{example m4}
Computations in {\tt Singular} for fixed $a,c$ over a finite field suggest that $\,\ch\left( I_4 \right)\,$ 
decomposes into $\, 51=B_5-1 \,$ irreducible components. One of them, $ K  \coloneqq  V(x_1-x_2,\,x_1-x_3,\,x_1-x_4)$, is $5$-dimensional. 
The analogous computations over $\, \Q(a,c)\,$ do not terminate. We can nevertheless verify its existence via the following trick. 
Instead of $I_4$, we consider the ideal $J_4 \coloneqq I_4+(x_1-x_2)$. Then we clearly have:
$$\ch(I_4) \,\supseteq \, \ch(I_4) \,\cap\, V(x_1-x_2) \, \supseteq \, \ch(J_4).$$
The computation of $\,\ch(J_4)\,$ is much simpler and immediately terminates. It turns out that $ K  \subseteq \ch(J_4)$. Therefore, $\,\ch(I_4)\,$ contains the $5$-dimensional component~$\,K\,$ and we conclude that $\,I_4\,$ is not holonomic.
\end{example}

\subsection{Open problems concerning the characteristic variety}\label{section open problems}
As the examples above indicated, there are a lot of open problems which we would like to put forward.

\begin{problem}
Compute the Weyl closure $\,W(I_m)\,$ of $\,  I_m \,$ for any $m$.
\end{problem}

\noindent A first step would be to explicitly write down differential operators in $W(I_m)\ohne I_m$.

\begin{problem}\label{problem group action}
Show that $\,\ch(W(I_m))\,$ (and possibly $\ch(I_m)$) are invariant under the action of $\,\mathbb{C}^{\ast}\times \mathbb{C}^{\ast}\,$ on $\,T^*\A^m=\A^m\times \A^m\,$ given by scalar multiplication on the factors.
\end{problem}

This would of course be an immediate consequence of a proof of Conjecture~\ref{conjecture charVar}. It should however be easier to tackle Problem \ref{problem group action} directly. One strategy could be to write down a flat one-parameter family of ideals $\{J_t\}_{t\in \A^1}$, such that $\,J_1=I_m\,$ and $\,J_0\,$ has an action by \mbox{$\,\mathbb{C}^{\ast} \times \mathbb{C}^{\ast}\,$} and then to see how to relate the characteristic varieties in a flat family. 

One way to realize such a one-parameter family concretely is to apply a suitable $\mathbb{C}^{\ast}$-action to $\,I_m\,$ and take the limit as the parameter $\,t\,$ of $\,\mathbb{C}^{\ast}\,$ goes to zero. If e.g.\ we decree the $\,x_i\,$ to have weight zero and the $\,\xi_i\,$ have weight one, the commutator relation of the Weyl algebra is preserved and for each $t$ we obtain an ideal $J_t$ as claimed. The flat limit is stable under the $\C^\ast$-action and can be found by applying the action to a Gröbner basis. Note that the action on $\,J_0\,$ induces an action of $\,{\mathbb{C}^{\ast} \times \mathbb{C}^{\ast}}\,$ on $\ch(J_0)$, as the latter always has a $\mathbb{C}^{\ast}$-action given by scalar multiplication on the fibers of $T^*\A^m \to \A^m$.

There are also other instances of annihilating ideals related by one-parameter families. It is classically known that the hypergeometric functions $\,{_0F_{\!\!\;1}}\,$ and ${}_1F_{\!\!\;1}$ are related to one another through a scaling and limit process. More precisely, ${_1F_{\!\!\;1}}(a;c)\big(\frac{1}{a}X\big) \to {_0F_{\!\!\;1}}(c)(X)$ as $ a\to \infty$, see~\cite[Section 7.5]{Mui82}. Also, the hypergeometric function $\,{_0F_{\!\!\;1}}\,$ is known to be annihilated by the operators
\begin{align}\label{ann0F1}
x_k\partial_k^2 \,+\,  c\partial_k \,+ \, \frac{1}{2}\left( \sum_{{\ell}\neq k} \frac{x_{\ell}}{x_k-x_{\ell}}(\partial_k - \partial_{\ell}) \right) \,-\, 1, 
\end{align}
where $ k =1,\ldots,m$. One directly checks that the $\,g_k\,$ from \eqref{ann1F1} scale accordingly to give the system \eqref{ann0F1}, see \cite[Theorem 7.5.6]{Mui82}.

\begin{problem}
Can the scaling relation between $\,{_0F_{\!\!\;1}}\,$ and $\,{}_1F_{\!\!\;1}\,$ be used to  deduce a relation between the characteristic varieties of $\, I_m\,$ and the corresponding ideal generated by the operators~\eqref{ann0F1}? 
\end{problem}

We would like to mention that ${_0F_{\!\!\;1}}$ naturally appears when investigating the normalizing constant of the Fisher distribution on $\text{SO(3)}$, as described in~\cite{SeiFisher}.

\subsection{Outlook}\label{section outlook}
We think that Conjecture~\ref{conjecture charVar} deserves further study and that it will be helpful to get a better understanding of the hypergeometric function $\,{_1F_{\!\!\;1}}\,$ of a matrix argument. The goal of the present article was to put forward this very clear and intriguing conjecture and to provide some evidence for it. The context in which we studied the function $\,{_1F_{\!\!\;1}}\,$ was rather conceptual, but our methods were mainly ad hoc. We believe that, eventually, the problem should be addressed using more advanced methods from $D$-module theory. For this, one should look for a more intrinsic description of the Muirhead ideal---or rather its Weyl closure. In particular, it would be interesting to understand if there is some generalization of GKZ systems and a relation to the hypergeometric function of a matrix argument similar to the one-variable case. We hope to be able to tackle these problems in the future.

\appendix

\section{Singular locus for special parameters} \label{sec:appendix}

In \cref{section solutions muirhead}, we discussed the singular locus of the Muirhead ideal and of its Weyl closure for those parameters $a,c$, for which the hypergeometric function $\,{_1F_{\!\!\;1}}(a;c)\,$ of a diagonal matrix argument is defined. In this appendix, we prove Theorem~\ref{thm:singularLocus} without any restriction on the parameter $c \in \C$. We are grateful to the referee for proposing an approach based on restriction modules, which finally led to the proof presented here. We would like to point out that similar problems have been studied in the literature. In \cite{HT14}, the singular locus of a holonomic system annihilating Lauricella's hypergeometric function $\,F_C\,$ was computed using a different technique. Hattori--Takayama used Gröbner bases and syzygies to compute a certain Ext-module whereas we analyze restriction modules on coordinate hyperplanes by a more elementary, computational argument. It would be interesting to compare the two methods more thoroughly.

Even though our approach in \cref{section solutions muirhead} for studying the singular locus rests only on the differential operators, defined regardless of the value of the parameters, the need to consider non-special parameters shows up in one subtle step of the computations: To prove that the coordinate hyperplanes $\, \{x \in \C^m \mid x_i = 0\}\,$ lie in the singular locus of $W(I_m)$, in \cref{lem:solAroundZeroComp} our proof relied on the condition \mbox{$c \notin \frac{m-1}{2} - \N$}. Note that this is the only step in the proof of \cref{thm:singularLocus} that does not work for arbitrary $c \in \C$. In particular, the diagonal hyperplanes were shown to lie in the singular locus for any $\,c\,$ by Lemma~\ref{lem:solAroundDiagComp}. Therefore, to prove \cref{thm:singularLocus}, it suffices (by symmetry) to show that the hyperplane $\, H  \coloneqq \{x \in \C^m \mid x_m = 0\}\,$ is contained in the singular locus of $W(I_m)$.

For this, we investigate the $D_{m-1}$-module $D_m/(W(I_m) + x_m D_m)$, which is the {\em restriction module} of $\,D_m/W(I_m)\,$ with respect to $H$. Its holonomic rank coincides with the dimension of the space of formal power series solutions to $\,W(I_m)\,$ centered at a general point of $H$.
Hence, recalling that $\,\hrk(W(I_m)) = 2^m\,$ by \cref{corollary hntt}, we can conclude \cref{thm:singularLocus} from the following result:

\begin{prop}
  Let $\,a,c \in \C\,$ be arbitrary parameters. Then the holonomic rank of the restriction module $\,D_m/(W(I_m) + x_m D_m)\,$ is strictly smaller than $2^m$.
\end{prop}

\begin{proof}
  Consider the localized Weyl algebra 
  \[D_{(x_m)} \, \coloneqq \, \C[x_1,\dots,x_m]_{(x_m)} \otimes_{\C[x_1,\dots,x_m]} D_m,\]
  which is the ring of differential operators with rational function coefficients that do not have poles along the hyperplane $\{x \in \C^m \mid x_m = 0\}$. Denote by $\,J_m\,$ the ideal $W(I_m) \cap D_{(x_m)} \subseteq D_{(x_m)}$. Then the inclusion $D_m \subseteq D_{(x_m)}$ induces an isomorphism of $R_{m-1}$-modules
  \[R_{m-1} \otimes_{D_{m-1}} D_m/(W(I_m)+x_mD_m) \, \cong \, D_{(x_m)}/(J_m + x_m D_{(x_m)}) \, \eqqcolon \, M.\]
  By definition, the holonomic rank of the restriction module $\,D_m/(W(I_m) + x_m D_m)\,$ is the dimension of the $R_{m-1}$-module $\,M\,$ as a vector space over $\C(x_1,\dots,x_{m-1})$. Therefore, our aim is to bound $\dim_{\C(x_1,\dots,x_{m-1})} M$.
  Note that $$D_{(x_m)}/x_m D_{(x_m)} \,\cong \, \C[x_1,\dots,x_m]_{(x_m)}/(x_m) \otimes_{\C[x_1,\dots,x_m]} D_m$$ is a free $R_{m-1}$-module isomorphic to $\, R_{m-1}^{\oplus \infty} \,$ with the countable basis $\{1,\partial_m, \partial_m^2,\dots\}$. For each operator $Q \in D_{(x_m)}$, we write $ \, \restr{Q}{x_m=0} \, $ for the unique expression $\, \sum_{i=0}^k Q_i \partial_m^{i}\,$ with $\,Q_i \in R_{m-1}\,$ representing it in $D_{(x_m)}/x_m D_{(x_m)}$. Then $\,M\,$ is the quotient of $\,D_{(x_m)}/x_m D_{(x_m)}\,$ by the $R_{m-1}$-submodule 
  \[N \, \coloneqq \, \{\restr{Q}{x_m=0} \mid Q \in J_m\}.\]
  We equip $\,D_{(x_m)}/x_m D_{(x_m)}\,$ with a total order $\,\prec\,$ on its $\C(x_1,\dots,x_{m-1})$-basis of monomials $\,\{\partial^\alpha = \partial_1^{\alpha_1} \dots \partial_{m-1}^{\alpha_{m-1}} \partial_m^{\alpha_m} \mid \alpha \in \N^m\}\,$ as follows: For $\alpha, \beta \in \N^m$,
  \[\partial^\alpha \prec \partial^\beta \ \  :\Leftrightarrow \ \  \alpha_m < \beta_m \text{ or } \big(\alpha_m = \beta_m \text{ and } (\alpha_1,\dots,\alpha_{m-1}) \prec_{\rm grlex} (\beta_1,\dots,\beta_{m-1})\big),\]
  where $\,\prec_{\rm grlex}\,$ denotes the graded lexicographic order on $\N^{m-1}$. This is a POT term order (“position over term”) on the free $R_{m-1}$-module $D_{(x_m)}/x_m D_{(x_m)} \cong R_{m-1}^{\oplus\infty}$, cf.~\cite[§5.2]{SST00}. The dimension of $\,M\,$ over $\,\C(x_1,\dots,x_{m-1})\,$ agrees with that of the associated graded module
  \[\gr^{\prec}(M) \, = \,  \C(x_1,\dots,x_{m-1})[\xi_1,\dots,\xi_{m-1}]^{\oplus \infty}/\init_\prec(N),\] where $\, \init_\prec(N)\, $ is the $\C(x_1,\dots,x_{m-1})[\xi_1,\dots,\xi_{m-1}]$-submodule generated by the initial forms of elements in $\,N\,$ with respect to $\prec$.
  
  Our approach is now to explicitly write out ($\prec$-initial forms of) elements in $N$ to bound the holonomic rank of the restriction module. 
  Note that the Muirhead operators $\,g_1,\dots, g_m\,$ from \eqref{ann1F1re} lie in $J_m$, hence $\,\restr{\partial_m^k g_i}{x_m = 0} \in N\,$ for all $\,i = 1,\dots, m\,$ and $k \in \N$. If $i \neq m$, one computes that
  \begin{equation} \label{eq:genOps}
    \restr{\partial_m^k g_i}{x_m = 0} \,=\, x_i \partial_i^2 \partial_m^k \,+\, \text{smaller  order terms w.r.t.} \prec.
  \end{equation}
  Hence, $x_i \xi_i^2 \, \partial_m^k \in \init_{\prec}(N)$ for all $i \leq m-1$, $k \in \N$.
  
  Moreover, a straightforward computation reveals that $\,  \restr{(\partial_m^k g_m)}{x_m = 0}  \, $ equals
  \begin{equation} \label{eq:spOps}
    \big(c+k - \frac{m-1}{2}\big) \partial_m^{k+1} \,-\, (k+a) \partial_m^k \,+\, \frac{1}{2} \sum_{\ell=0}^k \frac{k!}{\ell!} \sum_{j=1}^{m-1} x_j^{\ell-k} (\partial_j - \ell x_j^{-1}) \; \partial_m^\ell.
  \end{equation}
  In particular, we see that $\,\partial_m^{k+1} \in \init_{\prec}(N)\,$ for all \,$k \in \N\,$ with $c+k-\frac{m-1}{2} \neq 0$.
  
  For the case that $c \notin \frac{m-1}{2} - \N$, we conclude that $\,\gr^{\prec}(M)\,$ is generated over $\,\C(x_1,\dots,x_{m-1})\,$ by the $\,2^{m-1}\,$ elements $\,\xi_1^{\alpha_1} \cdots \xi_{m-1}^{\alpha_{m-1}}\,$ with $\alpha \in \{0,1\}^{m-1}$. This shows
  \begin{align*}
    \hrk(D_m/(W(I_m) + x_m D_m)) &\,=\, \dim_{\C(x_1,\dots,x_{m-1})}(M) \\ 
    &\,=\, \dim_{\C(x_1,\dots,x_{m-1})}\big(\gr^{\prec}(M)\big) \,\leq \, 2^{m-1} \,<\, 2^m,
  \end{align*}
  reproving \cref{lem:solAroundZeroComp}.
  
  Now, we turn to the remaining case $\,c = \frac{m-1}{2} - s\,$ for some $s \in \N$. In this case, \eqref{eq:genOps} for $\,k \in \{0, s\}\,$ and \eqref{eq:spOps} for $\,k \in \N \setminus \{s\}\,$ show that $\,\gr^{\prec}(M)\,$ is generated over $\,\C(x_1,\dots,x_{m-1})\,$ by the $\,2^m\,$ elements $\,\xi_1^{\alpha_1} \cdots \xi_{m-1}^{\alpha_{m-1}} \, \partial_m^r\,$ with $\,\alpha \in \{0,1\}^{m-1}\,$ and $r \in \{0,s+1\}$. It suffices to prove that there is a linear dependence among these generators. Then we can conclude that 
  \begin{align*}
    \hrk \big(D_m/(W(I_m) + x_m D_m)\big) \,=\, \dim_{\C(x_1,\dots,x_{m-1})}\big(\gr^{\prec}(M)\big) \,\leq \,2^m-1 \,<\, 2^m.
  \end{align*}
  In the following lemma, we leverage \eqref{eq:spOps} for $\,k = s\,$ to show the linear dependence, concluding the proof.
\end{proof}

\begin{lemma}\label{lem 8.3}
  Let $a \in \C$, $s \in \N$ and $c = \frac{m-1}{2} - s$. For each $r \in \{1,\dots,s+1\}$, there exists an element in $\,N\,$ of the form
  \begin{equation} \label{eq:goodOp}
    H_r \,=\,s(s-1)\cdots (s-r+1) \partial_m^r \,-\sum_{\tau \,\in\, \{0,1\}^{m-1}} q_\tau^\supind{r} \partial_1^{\tau_1} \cdots \partial_{m-1}^{\tau_{m-1}},
  \end{equation}
  where $\,q_\tau^\supind{r} \in \C(x_1,\dots,x_{m-1})\,$ and for each $\,r\,$ there is at least one $\,\tau\,$ with $q_\tau^\supind{r} \neq 0$.
  In particular (setting $r = s+1$), the elements $\,\big\{\partial_1^{\tau_1} \dots \partial_{m-1}^{\tau_{m-1}} \mid \tau \in \{0,1\}^{m-1}\big\}\,$ of $\,M\,$ are not linearly independent over $\C(x_1,\dots,x_{m-1})$.
\end{lemma}

\begin{proof}
  We fix $m$ and $s$, and for $\tau \in \{0,1\}^{m-1}$, we denote $|\tau| := \tau_1+\dots+\tau_{m-1}$. By induction on $r$, we prove the following:
  \begin{enumerate}[label=(\roman*)]
    \item $x_1^{r-|\tau|} q_\tau^\supind{r} \in \C[x_1,\dots,x_{m-1}]_{(x_1)}$ for all $\tau \in \{0,1\}^{m-1}$.
    \item $\restr{\bigg(x_1^r q_{(0,0,\dots,0)}^\supind{r}\bigg)}{x_1 = 0} = 0$.
    \item $\restr{\bigg(x_1^{r-1} q_{(1,0,\dots,0)}^\supind{r}\bigg)}{x_1 = 0} \in \frac{1}{2^{2r-1}} \Z \setminus \frac{1}{2^{2r-2}} \Z$.
  \end{enumerate}
  Note that the expressions in (ii) and (iii) are well-defined because of (i). In particular, condition (iii) guarantees that $q_{(1,0,\dots,0)}^\supind{r} \neq 0$.
  
  For $r = 1$, let $\,H_1\,$ be the negative of \eqref{eq:spOps} with $k = 0$, which is of the desired form with 
  \[q_\tau^\supind{1} \,=\, \begin{cases} -a &\text{if } |\tau| = 0, \\ 1/2 &\text{if } |\tau| = 1, \\ 0 &\text{if } |\tau| > 2.\end{cases}\]
  For the induction step, fix $\,k \geq 1\,$ and assume that the operators $\,H_r\,$ have been constructed with properties (i) to (iii) for $\,r \leq k\,$. We wish to construct $H_{k+1}$. For this, we start with a suitable multiple of the expression \eqref{eq:spOps} and reduce it with respect to $\,\prec\,$ modulo the expressions $\,H_r\,$ for $\,r \leq k\,$ and the expressions $\,\restr{g_i}{x_m=0}\,$ for $\,i < m\,$ from \eqref{eq:genOps}. To verify the desired properties, we do not get around carrying out the calculations. Explicitly, the following element arises after the reductions modulo only the expressions $\,H_r\,$ for $r \leq k$:
  \begin{align*}
    \widetilde{H}_{k+1} \, \coloneqq \,{} &-\prod_{i=0}^{k-1} (s-i) \cdot  \restr{\big(\partial_m^k g_m\big)}{x_m = 0} - (k+a)H_k \\
    &+ \frac{1}{2} \sum_{r=1}^k \frac{k!}{r!} \prod_{i=r}^{k-1} (s-i) \sum_{j=1}^{m-1} x_j^{r-k} (\partial_j - r x_j^{-1}) H_r.
  \end{align*}
  Note that in $\widetilde{H}_{k+1}$, all terms involving $\,\partial_m^r\,$ for $\,r \notin \{0,k+1\}\,$ cancel, the term $\,\partial_m^{k+1}\,$ only occurs with coefficient $s\cdots (s-r+1)$, and all other terms are of the form $\,p \partial_1^{\alpha_1} \dots \partial_{m-1}^{\alpha_{m-1}}\,$ with $\,p \in \C(x_1,\dots, x_{m-1})\,$ and $\alpha \in \{0,1,2\}^{m-1}$. We define $\,H_{k+1}\,$ as the expression obtained by further reducing $\,\widetilde{H}_{k+1}\,$ modulo the expressions $\,\restr{g_i}{x_m=0}\,$ for $i \leq m-1$. For this, note that $\,\restr{g_i}{x_m=0}\,$ for $\,i < m\,$ are  the Muirhead operators \eqref{ann1F1} in dimension $m-1$, and are in particular a Gröbner basis for $\,\prec\,$ by \cref{theorem holonomic rank im}.
  
  Denote by $\,\nu \colon \C(x_1,\dots,x_{m-1}) \to \Z \cup\{\infty\}\,$ the discrete valuation with valuation ring $\C[x_1,\dots,x_{m-1}]_{(x_1)}$, i.e.,
  \[\nu(p) \, \coloneqq \, \sup\left\{i \in \Z \mid x_1^{-i} p \in \C(x_1,\dots,x_{m-1})_{(x_1)}\right\}.\]
  With this notation, property (i) can be reformulated as $\nu(q_\tau^\supind{r}) \geq |\tau|-r$.
  
  In $R_{m-1}$, a reduction with respect to $\,\prec\,$ modulo the Muirhead operator $\,\restr{g_i}{x_m=0}\,$ for $\,i \leq m\,$ replaces
  \begin{equation} \label{eq:reductionModMuirhead}
    \partial_i^2 \, \mapsto \, x_i^{-1} \bigg(\Big(s-\frac{1}{2}+x_i-\frac{1}{2} \sum_{j \neq i}^{m-1} \frac{x_i}{x_i-x_j}\Big) \partial_i +\frac{1}{2} \sum_{j \neq i}^{m-1} \frac{x_j}{x_i-x_j} \partial_j + a\bigg).
  \end{equation}
  Applying this reduction to $\,p \partial^\alpha\,$ with $p \in \C(x_1,\dots,x_{m-1})$, $\alpha \in \{0,1,2\}^{m-1}$ yields only terms $\,p' \partial^{\alpha'}\,$ with $|\alpha'|-\nu(p') \leq |\alpha|-\nu(p)$, and equality can only hold for $i = 1$.
  Therefore, to prove property (i) for $H_{k+1}$, it suffices to show that 
  $\,\widetilde{H}_{k+1}\,$ has only terms $\,p \partial^\alpha\,$ with $|\alpha|-\nu(p) \leq k+1$. This can easily be seen by substituting \eqref{eq:spOps} and \eqref{eq:goodOp} (for $r \leq k$) into the definition of $\widetilde{H}$, and using that property (i) holds for $\,r \leq k\,$ by the induction hypothesis.
  
  We now turn to verifying properties (ii) and (iii). For this, we denote 
  \[\bar{q}_\tau^\supind{r} \, \coloneqq \, \restr{\bigg(x_1^{r-|\tau|} q_\tau^\supind{r}\bigg)}{x_1 = 0} \in \C(x_2,\dots,x_{m-1})\]
  for all $r \in \{1,\dots,k+1\}$. To determine $\bar{q}_\tau^\supind{k+1}$, we restrict our attention to those terms $\,p \partial^\alpha\,$ in $\,H_{k+1}\,$ for which the $\,|\alpha|-\nu(p)\,$ attains the maximum, namely $k+1$. As we have seen above in \eqref{eq:reductionModMuirhead}, for this purpose, the terms of $\,\widetilde H_{k+1}\,$ that get reduced modulo $\,\restr{g_i}{x_m=0}\,$ for $\,2 \leq i \leq m-1\,$ can be ignored, and it suffices to carry out the reductions of $\,\widetilde{H}_{k+1}\,$ modulo the single Muirhead operator $\restr{g_1}{x_m=0}$, which results in
  \[
  \widetilde{H}_{k+1}  + \frac{1}{2} \sum_{r=1}^k \frac{k!}{r!} \prod_{i=r}^{k-1} (s-i) x_1^{r-k} \sum_{\tau : \tau_1 = 1} q_\tau^\supind{r} x_1^{-1} \partial^{(0,\tau_2,\dots,\tau_{m-1})} \Big(\restr{g_1}{x_m=0}\Big). 
  \]
  Expanding this expression and dismissing all terms $\,p \partial^\alpha\,$ with $\,|\alpha|-\nu(p) < k+1\,$ or with $\alpha_i > 1$, one reads off for all $\,\tau \in \{0,1\}^{m-1}\,$ the recursion
  \begingroup \small
  \begin{align*}
    \bar{q}_\tau^\supind{k+1} \,=\,
    \begin{cases}
      \scriptsize
      \begin{aligned}
        &\frac{1}{2} \sum_{r=1}^k \frac{k!}{r!} \prod_{i=r}^{k-1} (s-i)
        \bigg(\Big(|\tau|-2r+s-\frac{1}{2}\Big) \bar{q}_\tau^\supind{r} + \bar{q}_{\tau-e_1}^\supind{r}\bigg) {}+{} \prod_{j=2}^{m-1} (\tau_j-1) \cdot \frac{1}{2} k! \prod_{i=0}^{k-1} (s-i)\\ 
      \end{aligned}
      &\text{ if } \tau_1 = 1, \\[1em]
      \scriptsize
      \begin{aligned}
        \frac{1}{2} \sum_{r=1}^k \frac{k!}{r!} \prod_{i=r}^{k-1} (s-i)
        \bigg(\Big(|\tau|-2r+s-1\Big) \bar{q}_\tau^\supind{r} + \frac{1}{2} \sum_{j : \tau_j = 1}\bar{q}_{\tau-e_j+e_1}^\supind{r}\bigg)
      \end{aligned}
      &\text{ if } \tau_1 = 0,
    \end{cases}
  \end{align*}
  \endgroup
  where $\,e_j \coloneqq (0,\dots,0,1,0,\dots,0) \in \N^{m-1}\,$ with entry $\,1\,$ at the $j$-th position. 
	
  In particular, we immediately see that $\bar{q}_{(0,0\dots,0)}^\supind{k+1} = 0$, as $\,\bar{q}_{(0,0\dots,0)}^\supind{r} = 0\,$ for $\,r \leq k\,$ by the induction hypothesis, proving property (ii). Now, considering the above formula for $\tau = e_1$, we get
  \[\bar{q}_{e_1}^\supind{k+1} \,=\, \frac{1}{2}\Big(1-2k+s-\frac{1}{2}\Big)\bar{q}_{e_1}^\supind{k} \,+\, \bigg[\text{summands for $r < k$}\bigg] \,+\, \frac{1}{2}k! \prod_{i=0}^{k-1} (s-i).\]
  From the induction hypothesis, we see that the unique term of smallest $2$-adic valuation in this expression is $\frac{1}{4} \bar{q}_{e_1}^\supind{k} \in \frac{1}{2^{2k+1}} \Z \setminus \frac{1}{2^{2k}} \Z$. This shows property (iii) and concludes the proof.
\end{proof}

\bibliography{literatur}
\bibliographystyle{abbrv}

\bigskip
\bigskip
\end{document}